\documentclass{article}
\usepackage{amsmath,amsfonts,amsthm}
\usepackage{bbm}    \usepackage[utf8]{inputenc}
\usepackage{hyperref}

\newcommand\Lin{\mathrm{Lin}}

\newcommand\diam{\mathrm{diam}}
\newcommand\vol{\mathrm{vol}}

\newcommand\Id{\mathrm{Id}}
\newcommand\id{\mathrm{id}}
\newcommand\R{\mathbb{R}}
\newcommand\N{\mathbb{N}}
\newcommand\eps{\varepsilon}
\newcommand{\dist}{\mathrm{dist}}

\newcommand{\Isom}{\mathrm{Isom}}
\numberwithin{equation}{section}
\newtheorem{theorem}{Theorem}[section]
\newtheorem{lemma}[theorem]{Lemma}
\newtheorem{definition}[theorem]{Definition}
\newtheorem{proposition}[theorem]{Proposition}
\newtheorem{corollary}[theorem]{Corollary}
\newtheorem{remark}[theorem]{Remark}
\newcommand{\inj}{\operatorname{inj}}
\newcommand{\lin}{\mathop{\mathrm{Lin}}}

\newcommand\cof{\mathop{\mathrm{cof}}}
\newcommand\Cof{\mathop{\mathrm{Cof}}}
\newcommand\Det{\mathop{\mathrm{Det}}}
\DeclareMathOperator{\Div}{div}
\newcommand\Tr{\mathop{\mathrm{Tr}}}

\usepackage{fancyhdr}
\newcommand\Exp{{\mathrm{Exp}}}
\newcommand{\tr}{\operatorname{tr}}
\newcommand\oldalpha{\eta}
\newcommand\oldp{r}
\newcommand\oldr{p}
\newcommand\tmpexp{b}
\fancypagestyle{first}{%
\fancyhf{} \fancyfoot[C]{\thepage}
}
\begin{document}
\thispagestyle{first}
\begin{center}
{ \Large Optimal rigidity estimates for maps of a compact Riemannian manifold to itself
\\[5mm]}
{\today}\\[5mm]
Sergio Conti$^{1}$, Georg Dolzmann$^{2}$, Stefan M\"uller$^{1,3}$ \\[2mm]
{\em $^1$ Institut f\"ur Angewandte Mathematik,
 Universit\"at Bonn,\\ 53115 Bonn, Germany }\\
 {\em $^{2}$ Fakult\"at f\"ur Mathematik, Universit\"at Regensburg,\\
  93040
   Regensburg, Germany}\\
{\em $^3$ Hausdorff Center for Mathematics,
 Universit\"at Bonn,\\ 53115 Bonn, Germany }\\

 \begingroup
 \def\thefootnote{}
\footnotetext{
S.M. would like to thank Cy Maor for a very inspiring discussion during the
HIM Trimester program 'Mathematics for complex materials' and for pointing out
the importance of the Riemannian Piola identity.

This work was partially funded by the Deutsche Forschungsgemeinschaft (DFG, German Research Foundation) {\sl via} project 211504053 - SFB 1060, and project 390685813, GZ 2047/1.
}
\endgroup

\bigskip

\begin{minipage}[c]{0.85\textwidth}\small
\begin{center} {\bf  Abstract}
\end{center}
Let $M$ be a smooth, compact, connected,  oriented Riemannian manifold and let $\imath: M \to \R^d$ be an isometric embedding. We show that a Sobolev map $f: M \to M$ which has the property that the differential
$df(q)$ is close to the set $SO(T_q M, T_{f(q)} M)$ of orientation preserving isometries (in an $L^p$ sense)
is already $W^{1,p}$ close to a global isometry of $M$. More precisely we prove
for $p \in (1,\infty)$  the optimal linear estimate
$$\inf_{\phi \in \mathrm{Isom}_+(M)} \| \imath \circ  f - \imath \circ \phi\|_{W^{1,p}}^p  \le C E_p(f)$$
where
$$ E_p(f)  :=  \int_M  \dist^p(df(q), SO(T_q M, T_{f(q)} M)) \, d\vol_M$$
and where $\mathrm{Isom}_+(M)$ denotes the group of orientation preserving isometries of $M$.

This extends the  Euclidean rigidity estimate of Friesecke-James-M\"uller [Comm. Pure Appl. Math. {\bf 55} (2002), 1461--1506]
to Riemannian manifolds. It also extends the  Riemannian stability result of
Kupferman-Maor-Shachar [Arch. Ration. Mech. Anal. {\bf 231} (2019), 367--408] for sequences of maps with $E_p(f_k) \to 0$ to an optimal  quantitative estimate.

The proof relies on the weak
 Riemannian Piola identity of Kupferman-Maor-Shachar, a uniform $C^{1,\alpha}$ approximation
through the harmonic map heat flow,
and a linearization argument which reduces the estimate to the well-known Riemannian version of Korn's inequality.

\\[4mm]
\mbox{}
\end{minipage}
\end{center}

\section{Introduction}

\subsection{Main result}
In Euclidean space,  maps that are almost isometric and almost orientation preserving enjoy the following rigidity property.
\begin{theorem}[\cite{FrieseckeJamesMueller2002-CPAM}, Theorem~3.1]
\label{th:euclidean_rigidity}   
Let $U \subset \R^n$ be open, bounded and connected, with Lipschitz
boundary. Then there exists a constant $C$, depending on $U$ and $n$, such that for every $f \in W^{1,2}(U; \R^n)$ there
exists a constant matrix $R \in SO(n)$ with
\begin{equation} \label{eq:euclidean_rigidity}
\int_U |df - R|^2 \, dx \le  C \int_U  \dist^2(df, SO(n)) \, dx.
\end{equation}
\end{theorem}
In other words, if $df$ is $L^2$ close to the \emph{set} $SO(n)$ of matrices, then it is $L^2$ close
to a single matrix.
The result can be extended to $L^p$ estimates (for $1 < p < \infty$) and Lorentz-spaces estimates,  see, for example,  \cite{ContiDolzmannMueller2014}.

Such rigidity estimates have a long history. The fact that $df \in SO(n)$ almost everywhere
implies that $df$ is constant is known as Liouville's theorem. For $C^3$ maps it is proved in
Liouville's paper \cite{Liouville1850}.  For Sobolev maps it follows from Reshetnyak's approach to
 quasiconformal maps. In fact, Reshetnyak also established a stability result. He showed that if
 the right hand side of \eqref{eq:euclidean_rigidity} goes to zero then the left hand side goes to zero \cite{Reshetnyak1967}.
 Indeed,  Reshetnyak obtained such stability estimates also for  almost conformal maps.

The first estimate with the optimal linear scaling appears in
 a fundamental paper of John \cite[Theorem IV, p.\ 410]{JohnCPAM1961}. He proved the estimate
for $C^1$ maps with $\dist(df, SO(n))$ sufficiently small or, more generally,  locally bilipschitz maps with
bilipschitz constant close to $1$. In this case, one obtains a bound for the BMO norm of $df$ and John's paper
actually marks the birth of the space BMO: the paper \cite{JohnNirenberg1961BMO}, which introduces the space $\mathrm{BMO}$ of functions of bounded mean oscillation
appears right after \cite{JohnCPAM1961}.
Regarding $L^p$ estimates, Kohn \cite{Kohn82} established optimal estimates for
$f - (Rx + \text{const})$ in $L^p$, but not for $df- R$.

Linearization of \eqref{eq:euclidean_rigidity} gives an estimate of the full
derivative (up to a constant skew-symmetric matrix) in terms of its symmetric part, i.e., Korn's inequality.
Thus  \eqref{eq:euclidean_rigidity} can be seen as the natural geometrically nonlinear version  of Korn's inequality.

\medskip

The estimate \eqref{eq:euclidean_rigidity} plays an important role in nonlinear elasticity, in particular
for the rigorous derivation of theories for lower-dimensional elastic objects, such as plates, shells, rods and beams
(see, for example,
\cite{FrieseckeJamesMueller2002-CPAM,FrieseckeJamesMueller2006}
and the surveys
\cite{Mueller2017Review,Lewicka2023Buch})
and the rigorous derivation of linear elasticity from nonlinear elasticity \cite{DalmasoNegriPercivale2002}.
Other applications include models of dislocations and grain boundaries, see for example
\cite {LauteriLuckhaus,GarroniFortunaSpadaro}.
\medskip

In view of the recent interest in noneuclidean elasticity as well as out of intrinsic geometric interest,
there has  been a lot of activity in extending  \eqref{eq:euclidean_rigidity} to a Riemannian setting.
In particular,  Kupferman, Maor and Shachar \cite{KupfermanMaorShachar2019}  have  obtained  a  corresponding compactness and stability result
for maps between Riemannian manifolds. Roughly speaking,  if $f_k: M \to N$  are
 maps between oriented $n$-dimensional
Riemannian manifolds, $1 < p < \infty$  and $\dist(df_k, SO(M,N))$
(see below for a precise definition of this expression) converges to
zero in $L^p$, then the sequence $f_k$ converges in a $W^{1,p}$ sense to an isometric immersion $f_0$.
The argument in \cite{KupfermanMaorShachar2019} is based on a subtle extension of Reshetnyak's compactness argument
in the Euclidean case, but it does not provide an explicit estimate how close $f_k$ is to $f_0$.

Here we show the optimal rigidity estimate for  maps from a compact
manifold
to itself.

\begin{theorem} \label{th:optimal_riemannian_rigidity}
 Let $M$ be a smooth, compact, oriented, connected $n$-dimensional Riemannian manifold,
 and
let $\imath : M \to \R^d$ be a smooth isometric embedding.
Let $ 1 < p < \infty$. Then there exists a constant
$C$, which depends on $M$, the embedding into $\R^d$,   and $p$, such that for each $f \in W^{1,p}(M; M)$
there exists an orientation preserving  isometry
$\phi: M \to M$ such that
\begin{align}  \label{eq:optimal_riemannian_rigidity}
  & \,  \int_M |\imath \circ  f - \imath \circ \phi|^p + |d ( \imath \circ f)  - d (\imath \circ \phi)|^p \, d\vol_M   \\
\le   & \,  C \int_M \dist^p(df, SO(M)) \, d\vol_M.
 \nonumber
\end{align}
\end{theorem}

Here by a compact manifold we mean a manifold which is compact as topological space
(some authors use the expression 'compact manifold without boundary' for such manifolds). The integrand on the right hand side of   \eqref{eq:optimal_riemannian_rigidity}
is defined for almost every $q$ in $M$ and given by
$$  \dist(df, SO(M))(q) := \dist(df(q), SO(T_q M, T_{f(q)} M)),$$
where the tangent spaces $T_q M$ and $T_{f(q)}$ are viewed as oriented
Euclidean spaces with the Riemannian scalar product, $SO(V,W)$ denotes the
set of linear and orientation preserving maps between oriented Euclidean spaces $V$ and $W$
and the distance is taken with respect to the Hilbert-Schmidt norm on $W \otimes V^*$, see
Section~\ref{se:linear_algebra}  below for a detailed description.
 By an orientation-preserving isometry, we mean a  diffeomorphism that preserves distance and orientation,
 see  Lemma~\ref{le:isometries}  below for an  equivalent definition and the fact that isometries are smooth.

 For the special case that $M$ is the round sphere $S^n$, embedded in $\R^{n+1}$, the estimate
 \eqref{eq:optimal_riemannian_rigidity}
 was recently shown by Alpern, Kupferman, and Maor \cite{Alpern-Kupferman-Maor23}
 by a clever reduction to the Euclidean estimate. Indeed, if $f: S^{n} \to S^n$ and
 $\tilde f$ denotes the one-homogeneous extension of $f$ to the annulus
  $B_2 \setminus {\overline B_{\frac12}}\subset\R^{n+1}$, then $\dist(d\tilde f, SO(n+1))(x)$ is comparable to $
  \dist(df, SO(S^n, S^n))(x/|x|)$ and the estimate follows from Theorem~\ref{th:euclidean_rigidity}  and the fact that the action of
  $SO(n+1)$ on $\R^{n+1}$ preserves $S^n$.
   Earlier,   Chen, Li, and Slemrod proved a Riemannian rigidity estimate \cite[Theorem 3.2]{Chen-Li-Slemrod22} for maps into $S^n$  where the constant
   $C$ depends on higher norms of $f$.

 \medskip

It is possible to state  an intrinsic version of the estimate
\eqref{eq:optimal_riemannian_rigidity}  which does not involve the isometric embedding $\imath: M \to \R^d$,
for example, by using the Sasaki distance  on the tangent bundle
$TM$,
but we prefer to view $M$ as embedded into $\R^d$  to emphasize the analogy with the Euclidean estimate.
Also various generalizations  of
 Theorem~\ref{th:optimal_riemannian_rigidity} are possible, but here we focus
on the simplest case to avoid technicalities and to emphasize the
strong similarity between  the arguments for the Euclidean and the Riemannian case.

\subsection{Strategy of proof}

We first recall the key steps in the proof of the Euclidean rigidity result, Theorem~\ref{th:euclidean_rigidity}.
It is easy to see that it suffices to show the result for maps $f$ for  which the right hand side of
\eqref{eq:euclidean_rigidity} is small.
In the following, we refer to the right hand side of \eqref{eq:euclidean_rigidity} or
\eqref{eq:optimal_riemannian_rigidity} as the energy of the map.
The proof consists of three steps:
\begin{enumerate}
\item Lipschitz approximation: \quad By a 'truncation of gradients'  argument, see
  \cite{liu1977luzin}, \cite{AcerbiFusco1988approximation}, \cite{EvansGariepy1992},
one easily sees that it suffices to show the result under the additional condition that $|df| \le L$
almost everywhere where $L$ is a fixed constant, depending only on the dimension $n$.
\item Almost harmonicity and compactness:   \quad It follow from the Piola identity $\Div \cof df = 0$
and the pointwise estimate $| F - \cof F| \le C \dist(F, SO(n))$  for all $F$ with $|F| \le L$
that low energy maps $f$ are almost harmonic. Thus $f$ can be written as the  sum of a harmonic part
$u$ and a map $w$ which is controlled in $W^{1,2}$ by the energy of $f$.
The harmonic part enjoys additional regularity and compactness properties,  and it is easy to see
that $u$ is $C^1$ close to an affine isometry $\psi$, locally, i.e. on open subsets for which the closure
is contained in $\Omega$. Since the energy is invariant under
left composition by $\psi^{-1}$,  we can assume without loss of generality $u$ is locally $C^1$  close to the identity.
\item Linearization of the metric deficit: \quad If
$F$ is close to the identity matrix  then $\dist(F, SO(n))$ is comparable to the metric deficit
$F^T F - \Id$. The linearization of the equation $ (df)^T df - \Id = h$
is given by  $(df) + (df)^T = h$. The kernel of the elliptic operator $\mathcal L f = (df) + (df)^T$ consists of
the finite dimensional space of infinitesimal affine  isometries $x \mapsto W x + a$ with $W \in
\mathrm{skw}(n)$, the space of skew-symmetric matrices. Since $\mathrm{skw}(n)$ is the Lie algebra
of the Lie group $SO(n)$, an easy perturbation argument together with Korn's inequality
gives the estimate for $df$, up to the action of $SO(n)$.
\end{enumerate}

Strictly speaking,  the above outline leads to  interior estimates in the Euclidean setting. An additional covering argument is used to get estimates up to the boundary. No such argument is needed for Theorem~\ref{th:optimal_riemannian_rigidity}
since we work on a compact manifold.

\medskip

The strategy for the proof of the Riemannian rigidity result
in Theorem~\ref{th:optimal_riemannian_rigidity}
is the same.
For the Lipschitz approximation,
one has to ensure that the approximation of $\imath \circ f$ still maps to $\imath(M)$.
This can be easily achieved by a suitable projection, see  \cite[pp. 390--392]{KupfermanMaorShachar2019} or
Section~ \ref{se:Lip_approximation} below.

For the almost harmonicity and compactness argument one can replace the Euclidean
Piola identity  $\Div \cof df = 0$ by  the  extrinsic Riemannian Piola identity of
\cite{KupfermanMaorShachar2019}, see Theorem~\ref{th:extrinsic_piola} below.
This shows that low energy maps are almost harmonic maps.
We could use a local approximation by  harmonic maps to get local compactness.
For variety,  and in order to get directly a global result,  we use instead the extrinsic harmonic map heat flow
$t \mapsto F_t$ with initial datum $F_0 = \imath \circ f$. Then it is easy to see that for Lipschitz $F_0$
and  a suitable time
$\tau$ of order $1$ the maps $F_\tau$ satisfy a uniform $C^{1,\alpha}$ bound and that
the difference $F_\tau - F_0$ is controlled in $W^{1,p}$ by the energy, i.e., the right hand side of
\eqref{eq:optimal_riemannian_rigidity},  see  Theorem~\ref{th:heat_flow_close_to_Phi}   below.

The linearization argument for maps that are close to the identity in $C^1$ is very similar.
The metric deficit is now given by $f^*g - g$. If we write $f = \exp X$ then the linearization of the metric deficit  is given by the linear operator
$\mathcal L X := \nabla X + (\nabla X)^T$, where $\nabla$ denotes the covariant derivative corresponding
to the Levi-Civita connection.
The kernel   $\mathcal K$ of $\mathcal L$ is again finite dimensional and  consists
of  the so-called  Killing fields. By a classical result, $\mathcal K$ can be identified with the Lie
algebra of the Lie group $\mathrm{Isom}_+(M)$ of orientation preserving isometries.
Then we can conclude as before by elliptic estimates for $\mathcal L$ and  a perturbation argument,
see Section~\ref{se:linearization_metric_deficit} below, and in particular
Theorem~\ref{th:C1_rigidity_new} and Corollary~\ref{co:C1_rigidity_maps}.

\medskip

Conceptually, there are two elliptic systems at play. First, the harmonic map equation, for which a good
regularity exists as soon as the local oscillation is small, but for which the linearized kernel is in general too large. Secondly, the equation $f^*g - g = h$ for the metric deficit. This equation has the optimal linearized
kernel, but it is only useful for maps $f$ which are already  $C^1$  close to the identity, or an isometry.
Roughly speaking we use the harmonic map equation to show that we are $C^1$ close to an isometry
(but not with the optimal rate) and then use the metric deficit equation to conclude.

\section{Preliminaries and notation}
\label{secpreliminaries}
Throughout this paper, we assume that
$$ \text{$(M,g)$ is a smooth, connected, compact $n$-dimensional Riemannian manifold,}$$
unless explicitly stated otherwise.
We usually write just $M$ instead of $(M,g)$.
We fix an isometric embedding
$\imath: M \to \R^d$ and we denote by $d(\cdot,\cdot)$ the inner metric
of $M$, i.e. $d(q,q')$ is the length of the shortest geodesic connecting $q$ and $q'$.
We use standard notation for Riemannian manifolds, see, for example,
\cite{Cheeger-Ebin,Kobayashi-Nomizu63,Kobayashi-Nomizu69,ONeill83}.
In particular, we  use the Levi-Civita connection on $M$ and we denote the corresponding covariant
derivative by $\nabla$.
We use the summation convention, i.e., we sum over repeated indices, unless noted otherwise.
We identify $T_p\imath(M)$ with a subspace of $\R^d$
and $T_p\R^d$ with $\R^d$.

\medskip

In this section we first quickly review some linear algebra and the definition of the second fundamental
form of $\imath(M)$, to give a precise definition of the quantities which appear in the statement of the main
result and in the  Riemannian Piola identity. Then we recall the definition of Sobolev spaces of maps
with values in a manifold and a notion of distance of such Sobolev maps.

\subsection{Some linear algebra}  \label{se:linear_algebra}
Let $V$ be an $n$-dimensional vector space, let $\alpha$ be a non-trivial $n$-form on $V$ and let $A: V \to V$ be a linear map. Then one  defines the determinant $\Det A$  of $A$ by $A^*\alpha = \Det A \, \, \alpha$. The definition is independent of the choice of $\alpha$ since the $n$-forms on $V$ form  a one-dimensional vector space. The map $A \to \Det A$ is a polynomial and we define the trace of $A$ by
\begin{equation} \Tr A := \frac{d}{dt}|_{t=0} \Det(\Id + t A).
\end{equation}
It is easy to see that for any matrix $M_A$ which represents $A$ with respect to a basis of $V$ we have
\begin{equation}  \label{eq:det_tr_V}
\Det A = \det M_A, \quad \Tr A = \tr M_A
\end{equation}
where $\det$ and $\tr$ denote the usual determinant and trace for $n \times n$ matrices, or, equivalently,
for maps from $\R^n$ to $\R^n$.
If $V$ has an inner product, one can also define the transpose $A^T$ and the cofactor operator $\Cof A$.
We now extend these notions for maps between two oriented  inner product spaces.

Let $V$ be an $n$-dimensional inner product space. We make $V$ an oriented space by fixing an $n$-form
$\alpha_0 \in \Lambda^n V$ with $\alpha_0 \ne 0$. We say that  a basis $v_1, \ldots, v_n$ of $V$ is positively
oriented if  $\alpha_0(v_1, \ldots, v_n) > 0$. Let $\tilde v_1, \ldots, \tilde v_n$ be a positively oriented orthonormal
basis of $V$. Then there exists a unique $n$-form $\alpha \in \Lambda^n V$ such that
$\alpha(\tilde v_1, \ldots, \tilde v_n) = 1$. It is easy to see that $\alpha(\hat v_1, \ldots, \hat v_n) = 1$ for every
other positively oriented orthonormal basis  $\hat v_1, \ldots, \hat  v_n$. We call $\alpha$ the volume form
of the oriented inner product space $V$ and denote it by $\vol_V$.

The inner product on $V$ defines
canonical inner products
on the dual space of one forms $\Lambda^1 V = V^*$, the space of $k$-forms $\Lambda^k V$, and
 the space of $(r,s)$-tensor $V^{\otimes r} \otimes (V^*)^{\otimes s}$. Given another oriented inner product space $W$ we
 also get a canonical  inner  product on spaces of mixed tensors like  $W^{\otimes r} \otimes (V^*)^{\otimes s}$. We will
 often deal with the space of linear maps
 $$\lin(V,W) \simeq W \otimes V^*$$
  and to fix notation, we spell out some details in this setting.

 \begin{definition} Let $V,W$ be $n$-dimensional, oriented, inner product spaces with volume forms $\vol_V$ and
 $\vol_W$, respectively. For $A, B \in \lin(V,W)$ we define the transpose  map $A^T \in \lin(W,V)$,  the scalar product, and
 the (Hilbert-Schmidt) norm  by
 \begin{equation}   \langle v, A^T w \rangle_V = \langle Av, w\rangle_W \quad \text{for all $v \in V$, $w \in W$,}
 \end{equation}
 \begin{equation} \langle A, B \rangle_{V,W} := \Tr A^T B, \quad |A|_{V,W} := \sqrt{\langle A,  A\rangle_{V,W}}.
 \end{equation}
 We define $\Det A \in \R$ by
 \begin{equation}
 A^* \vol_W = \Det A \, \vol_V.
 \end{equation}
Then $t \mapsto \Det (A + tB)$ is a polynomial and we define the cofactor operator $\Cof A \in \lin(V,W)$ by
\begin{equation}  \label{eq:define_Cof}
\langle \Cof A, B \rangle_{V,W} := \frac{d}{dt}|_{t=0} \Det(A + tB).
\end{equation}
We define the set $SO(V,W)$ of orientation-preserving isometries by
\begin{equation}
SO(V, W) := \{ A \in \lin(V,W) : A^T A = \Id_V, \, \Det A = 1\}
\end{equation}
and  set
\begin{equation}   \label{eq:dist_SO}
\dist(A, SO(V,W)) := \min \{ |A - B|_{V,W} : B \in SO(V,W) \}.
\end{equation}
\end{definition}

The intrinsically  defined quantities $\Det A$ and $\Cof A$ can be computed from the matrix $M_A$,
obtained by expressing  $A$ in oriented
orthonormal bases of $V$ and $W$.

\begin{lemma}\label{lemmalinearalgebra1}
Let $v_1, \ldots, v_n$ and $w_1, \ldots, w_n$ be positively oriented orthonormal bases
of $V$ and $W$, respectively. Let $A, B \in \Lin(V,W)$ and let  $M_A$ and $M_B$ be the
matrices representing $A$ and $B$ with respect to these bases, i.e. $A v_\beta = (M_A)^\alpha_\beta w_\alpha$ and  $B v_\beta = (M_B)^\alpha_\beta w_\alpha$.  Then
\begin{eqnarray}
M_{A^T}  &=& (M_A)^T,  \label{eq:MA_transpose}\\
\langle A, B \rangle_{V,W} &=&  \langle A v_\beta, B v_\beta \rangle  =
(M_A)^\alpha_\beta (M_B)^\alpha_\beta = \langle M_A, M_B \rangle_{\R^n, \R^n} \, ,
\label{eq:MA_scalarproduct}\\
\Det A &=& \det M_A,   \label{eq:MA_det} \\
(\Cof A)  v_\beta &=& (\cof M_A)^\alpha_\beta w_\alpha.  \label{eq:MA_cof}
\end{eqnarray}
\end{lemma}

\begin{proof}   \eqref{eq:MA_transpose}: Since the bases are orthonormal
we have $(M_{A^T})^\alpha_\beta  = \langle v_\alpha, A^T w_\beta\rangle = \langle A v_\alpha, w_\beta\rangle = (M_A)^\beta_\alpha$.\\
\eqref{eq:MA_scalarproduct}: We have  $(A^T B) v_\beta = \langle (A^T B) v_\beta, v_\alpha\rangle v_\alpha =
\langle B v_\beta, A v_\alpha\rangle v_\alpha$. Hence the first identity follows from the second identity in
  \eqref{eq:det_tr_V}, applied to $A^T B$. The second identity follows from $\langle w, w'\rangle_W =
  \langle w, w_\alpha\rangle_W  \langle w', w_\alpha\rangle_W$ with $w = B v_\alpha$, $w' = A v_\alpha$. \\
     \eqref{eq:MA_det}: We have $\vol_V(v_1, \ldots, v_n) = 1$ and writing $M := M_A$ we get\\
     $(A^*\vol_W)(v_1, \ldots, v_n)
      =  M^{\alpha_1}_1 \ldots M^{\alpha_n}_n \alpha_W(w_{\alpha_1}, \ldots, w_{\alpha_n})
      = \det M$,  since\\
       $\vol_W(w_1, \ldots, w_n) = 1$ and $\vol_W$ is alternating.\\
      \eqref{eq:MA_cof}: This follows from   \eqref{eq:MA_det}, \eqref{eq:MA_scalarproduct}, and the identity
      $\frac{d}{dt}|_{t=0} \det (M + tM') =\langle\cof M, M'\rangle_{\R^n, \R^n}$ for matrices $M$ and $M'$.
\end{proof}
It follows that $\Cof$ and $\Det$ inherit the properties of $\cof$ and $\det$ on matrices. In particular,
\begin{eqnarray}   \label{eq:props_Det}
\Det(AB) &=& \Det A \, \Det B,  \quad  \Det A^T = \Det A, \\
\Cof(AB) & = &  \Cof A \, \Cof B, \quad A^T \Cof A = \Id \, \Det A  \label{eq:props_Cof}, \\
 \label{eq:CofF_equals_F}
\Cof F &=& F \quad \text{if $F \in SO(V,W)$.}
\end{eqnarray}

We also have
 \begin{equation}   \label{eq:operatornorm_hilbert_schmidt}
 |A v|_W \le |A|_{V,W} \, |v|_V    \quad \text{ for all $A \in \lin(V,W)$ and $v \in V$.}
 \end{equation}
Indeed, if $v \ne 0$ we can choose a positively oriented orthonormal basis with $v_1 = v/ |v|$
and apply the first identity in \eqref{eq:MA_scalarproduct} with $B=A$.

\medskip

When the spaces $V$, $W$ and their scalar products are clear from the context we will often simply write
$\langle A, B \rangle$ instead of $\langle A, B \rangle_{V,W}$ and similarly for the norm.
Since $\Lin(V,W) \simeq W \otimes V^*$ we also also use the notation
\begin{equation} \label{eq:notation_scalar_product}
\langle A, B \rangle_{W \otimes V^*} := \langle A,B \rangle_{V,W}.
\end{equation}

\medskip

If $(M,g)$ and $(N,h)$ are Riemannian manifolds,  we will  use the above
definition for maps $A \in \lin(T_q M, T_{q'} N) \simeq T_{q'} N \otimes T_q^* M$
where the tangent spaces are equipped with the inner product given by the Riemannian metrics $g(q)$ and $h(q')$
and their orientation is induced by the orientation of $M$ and $N$.
For a  (weakly) differentiable map $f: U \subset M \to N$ we use the shorthand notation
\begin{equation} \dist(df, SO(g,h))(q) := \dist(df(q), SO(T_q M, T_{f(q)} N))
\end{equation} where the tangent spaces are equipped with the Riemannian metrics $g(q)$ and $h(f(q))$.
If $(N,h)=(M,g)$ we write
\begin{equation}  \dist(df,SO(M)) := \dist(df, SO(g,h)).
\end{equation}

\medskip

In \cite[Definition 1 and 2]{KupfermanMaorShachar2019} $\Det A$ and the cofactor operator $\Cof A$ are defined using the
Hodge-* operators in $V$ and $W$. It follows from  \cite[Proposition 2]{KupfermanMaorShachar2019} that our definition of
 $\Det A$ agrees with
theirs. Furthermore, it follows from the matrix representations of $A^T$, $\Det A$ and $\Cof A$ in
\cite[Lemma 4]{KupfermanMaorShachar2019} that our definition of $\Cof A$ agrees with theirs.

\subsection{Second fundamental form}
\label{secsecondfundform}
Recall that  we fix an isometric embedding $\imath: M \to \R^d$, where $\R^d$ is equipped with the
standard Euclidean metric $\mathfrak e$.
For $q \in \imath(M)$ we denote by $N_q\imath (M)$ the  normal space at $q$, i.e., the subspace of $\R^d$ perpendicular
to $T_q\imath (M)$.
We write $w^\perp$ for the  projection of a vector $w \in \R^d$ onto the space $N_q\imath(M)$.

Let $Y : \imath(M) \to T\imath(M)$ be a tangential  vector field, i.e., $Y(p) \in T_p\imath(M)$ for all $p$,
and let $v \in  T_q\imath(M)$. Let $d_v Y$ denote the directional derivative
in $\R^d$ (this agrees with the covariant derivative $\nabla_v$  in $\R^d$).
Using a local
frame of $T\imath(M)$,  it is easy
 to see that $(d_v Y)^\perp$, the normal component at $q$, depends only on $Y(q)$.
We define
the second fundamental form $A(q) : T_q\imath(M) \times T_q\imath(M) \to N_q\imath(M)$
 by
\begin{equation}
\label{eq:definition_second_fundamental}
 A(q)(v, w) := - (d_v Y)^\perp,  \quad \text{if $Y(q) = w$.}
\end{equation}
Here we choose the sign of $A$ consistent with \cite[p.\ 381]{KupfermanMaorShachar2019}, 
\cite[p.\ 216]{Struwe_variational_methods2}, \cite[p.\ 2]{Shatah-Struwe98},  or
\cite[eqn.\ (1.8)]{Lin-Wang08}, 
see  \eqref{eq:2nd_form_normal}.
 In \cite[Chapter VII]{Kobayashi-Nomizu69}  or in \cite[Item (2.21)]{EellsLemaire1978report}
the opposite sign is chosen.
It is easy to see that $A$ is symmetric,
see \cite{Kobayashi-Nomizu69}, Chapter VII, Proposition 3.2. 
The second fundamental form can equivalently be expressed in terms of the derivative of a normal field
$\nu: \imath(M) \to N\imath(M)$, i.e., a field with  $\nu(p) \in N_p\imath(M)$ for all $p$. Indeed,  differentiation of the
identity $\langle Y, \nu\rangle = 0$ for a tangential vector field $Y$  in direction $v \in T_q\imath(M)$ gives
$\langle Y(q),d_v \nu(q) \rangle = - \langle (d_v Y)(q), \nu(q) \rangle = \langle A(q)(v, Y(q)), \nu\rangle$.
Choosing a local orthonormal frame $\nu_1, \ldots, \nu_{d-n}$ of the normal bundle,  we get
\begin{equation}  \label{eq:2nd_form_normal}
A(q)(v,w) = \langle w, d_v \nu_i(q)\rangle \nu_i(q)  \quad \text{for $v,w \in T_q\imath(M)$.}
\end{equation}

The second fundamental form is also closely related to the second derivative of the  smooth closest point
  projection $\pi: \mathcal O \subset \R^d \to \imath(M)$, defined in a small neighborhood $\mathcal O$ of $\imath(M)$. Indeed, for $q \in \imath(M)$, the differential
  $d\pi(q)$ is the
  orthogonal projection from $\R	^d$ to $T_{q}\imath(M)$.
  Thus differentiating the identity $\langle d\pi(\gamma(t)) z-z,d\pi(\gamma(t))  w\rangle =0$
  for $z \in \R^d$, $w \in  T_{q}\imath(M)$ and
    a curve $\gamma: (-\delta, \delta) \to \imath(M)$ with $\gamma'(0) = v \in T_{q}\imath(M)$ and applying the definition of $A$ with $Y(t) = d\pi(\gamma(t)) w$ we get
  \begin{align}  \label{eq:second_derivative_pi}
  &  \langle d^2\pi(q)(v,z), w\rangle = \langle d\pi(q)z-z, d^2\pi(q)(v,w)\rangle =-  \langle
 z, A(q)(v,w)\rangle  \\
  &  \text{for all $v,w \in T_{q}\imath(M)$ and all $z \in \R^d$.}  \nonumber
  \end{align}
 We extend $A$ to a  symmetric map on linear maps from $T_p M$ to $T_{\imath(p)}\imath(M)$ as follows.
 Define
  $\mathbb A:  (T_{\imath(p)} \imath(M)  \otimes T^*_p M) \times  (T_{\imath(p)} \imath(M)  \otimes T^*_p M)
   \to N_{\imath(p)} \imath(M)$ by
   \begin{equation} \label{eq:definition_second_endos}
 \mathbb A(\imath(p))(X, Y) := \Tr A(X,Y) := g^{\alpha \beta} A(X e_\alpha, Y e_\beta)
\end{equation}
where $e_1, \ldots, e_n$ is a basis of $T_p M$ and $g^{\alpha \beta}$ is the inverse of
$g_{\alpha \beta} := g(e_\alpha, e_\beta)$.
In \cite{KupfermanMaorShachar2019} also the extended second fundamental form is denoted by $A$.

\bigskip

\subsection{Sobolev spaces on manifolds}\label{secsobolev}
The Sobolev space $W^{1,p}(M)$ of maps $u: M \to \R$ can be defined using local charts.
Equivalently, $u \in W^{1,p}(M)$ if there
 exists an $L^p$ section $\gamma$ in the cotangent bundle
$T^*M$ such that,
 for every $C^1$ section $\varphi$ of $T^*M$,
\begin{equation}  \label{eq:weak_derivative_M}  \int_M \langle \gamma, \varphi  \rangle \, d\vol_M = - \int_M u \, \delta \varphi \, d\vol_M,
\end{equation}
where $\delta$ is the codifferential. We call $\gamma$ the weak differential of $u$ and denote it by $du$.

 The Sobolev space $W^{1,p}(M; \R^d)$ and the weak differential of a map
 $u \in W^{1,p}(M; \R^d)$ are defined componentwise.  If $N$ is a smooth manifold and $\imath:N\to\R^d$ is an isometric embedding, we define
 \begin{align}\label{eqdefsobolevMN}
& \,   W^{1,p}(M; N)
 :=   \, \{ f:M\to N: \imath\circ f \in W^{1,p}(M; \R^d)  \}.   \end{align}
 If $f \in W^{1,p}(M; N)$,  then it is easy to see that for almost every $q \in M$ the weak differential $df(q)$ is a map
 from $T_q M$ to $T_{f(q)}N $. One can use, for example, the fact that $\imath\circ f$ is approximately differentiable
 almost everywhere (as a map with values in $\R^d$) and that the approximate differential and the weak differential agree
 almost everywhere.
 Thus, for almost every $q \in M$ one can define $df(q) \in T_{f(q)}N \otimes T^*_qM$ as the unique
 element of $T_{f(q)}N \otimes T^*_qM$ such that $d\imath(f(q)) df(q) = d(\imath \circ f)(q)$.

  For  an intrinsic definition of  $W^{1,p}(M; N)$, see  \cite{ConventVanSchaftingen2016intrinsic}. One can also define
 $ W^{1,p}(M; N)$ using the theory of metric-valued Sobolev spaces, see, for example, \cite{HeinonenKoskelaShanmugalingamTyson2015Buch}.

 \medskip

For $f$, $g \in W^{1,p}(M;M)$ we define
a distance by
\begin{equation}  \label{eq:dist_W1p(M,M)}
d_{1,p}(f,g) := \| \imath \circ  f- \imath \circ g\|_{W^{1,p}}.
\end{equation}
The following result shows that this distance behaves well under the action of isometries.

\begin{lemma}    \label{le:distance_TM_via_imath}
There exists a constant $C$ with the following properties.
\begin{enumerate}
\item\label{le:distance_TM_via_imathp}
For $p,q \in M$,
\begin{equation} \label{eq:equivalence_extrinsic_distance}
C^{-1} d(p,q) \le |\imath(p) - \imath(q)|_{\R^d} \le d(p,q).
\end{equation}
\item\label{le:distance_TM_via_imathder}
For $v \in T_p M$ and $w \in T_q M$ set
\begin{equation} \label{eq:distance_tangent_bundle}
d_{TM}(v,w) :=  |\imath(p) - \imath(q)|_{\R^d}    + |d\imath(v)  - d\imath(w)|_{\R^d}.
\end{equation}
Then, for every isometry $\phi: M \to M$,
\begin{equation}  \label{eq:quasiinvariant_dTM}
     C_{v,w}^{-1} d_{TM}(v,w)               \le    d_{TM}(d\phi(v),d\phi(w))  \le  C_{v,w}  d_{TM}(v,w)
\end{equation}
with
\begin{equation}
C_{v, w} := C\big( 1 + \min( |d\imath(v)|,  |d\imath(w)|)  \big).
\end{equation}
\item\label{le:distance_TM_via_imathw1p}
If $p\in[1,\infty)$, $f \in W^{1,p}(M;M)$ and $g \in W^{1,\infty}(M;M)$ with $\| d(\imath \circ g)\|_{L^\infty} \le L$,
then, for every isometry $\psi: M \to M$,
\begin{equation}  \label{eq:dfg_quasiinvariant_isometry}
d_{1,p}(\psi \circ f, \psi \circ g) \le C(1+L) d_{1,p}(f,g),
\end{equation}
where $d_{1,p}(f,g)$ is defined by  \eqref{eq:dist_W1p(M,M)}.
\end{enumerate}
\end{lemma}

\begin{proof}  \ref{le:distance_TM_via_imathp}: Since $\imath$ is an isometric immersion,  the upper bound in \eqref{eq:equivalence_extrinsic_distance} follows by looking at the length of curves.
To show the lower bound, let $P_{\imath(p)} : \R^d \to T_{\imath(p)}\imath (M) \subset \R^d$ denote the orthogonal projection.
Then $P_{\imath(p)} \circ \imath$ is a smooth
map from $M$ to $T_{\imath(p)}\imath (M)$,
and $d(P_{\imath(p)} \circ \imath)(p)$ is an isometric linear map from $T_p M$ to the $n$-dimensional space $P_{\imath(p)}(\R^d)=T_{\imath(p)}\imath (M)$. Hence $P_{\imath(p)} \circ \imath$ is a smooth diffeomorphism  on a ball $B_{\delta(p)}(p)$.
In particular we can choose $\delta(p)$ so small that
 $(P_{\imath(p)} \circ \imath)^{-1} \circ P_{\imath(p)}$ has
Lipschitz constant at most $2$.

Since $M$ is compact there exists a $\delta > 0$ such that $\delta(p) \ge \delta$ for all $p \in M$.
Hence $d(p,q) \le 2 |\imath(p) -\imath(q)|$ if
$d(p,q) \le \delta$. Since $\imath$ is an embedding, compactness of $M$ also implies that there exists a $\delta' >0$ such that $|\imath(p) - \imath(q)|\ge \delta'$ for all $p,q$ with $d(p,q) \ge \delta$. Thus
 \eqref{eq:equivalence_extrinsic_distance} holds with $C := \max(2, \delta/\delta')$.

 \medskip

\ref{le:distance_TM_via_imathder}: We first note that it suffices to prove the upper bound for $d_{TM}(d\phi(v), d\phi(w))$  in   \eqref{eq:quasiinvariant_dTM}. Indeed, since $\imath$ and $\imath \circ \phi$ are isometric immersions we
 have
 $$ |d(\imath \circ \phi)(v) | = |v| = |d\imath(v)|.$$
 Thus  applying the upper bound to $\tilde v$, $\tilde w$, $\tilde \phi$ instead of $v,w,\phi$  with
 $\tilde v  = d\phi(v)$, $\tilde w = d\phi(w)$, and $\tilde \phi = \phi^{-1}$, we get the lower bound.

 To prove the upper bound,  we may assume that $|d\imath(w)|  \le |d\imath(v)|$ since the assertion
 is symmetric in $v$ and $w$.
  Set $p' := \imath(p)$, $p'' := (\imath \circ \phi)(p)$, $v' := d\imath(v) $,
 $v'' := d(\imath \circ \phi)(v)$, and similarly for $q$ and $w$. 
Let $P$ and $\delta > 0$ be as in the proof of assertion \ref{le:distance_TM_via_imathp}.
 Then $\varphi_p := P_{p'} \circ \imath$ is a smooth diffeomorphism from $B_\delta(p)$ to its image
 and $d\varphi_p(p)$ is an isometry from $T_p M$ to $T_{p'}\imath(M)$, viewed as a subspace of $\R^d$.

 Now assume first that $d(p,q) < \delta$.
 Let $\eta:= \imath \circ \phi \circ \varphi_{p}^{-1}$.
 Then $\imath \circ \phi = \eta \circ P_{p'} \circ \imath$.
 Thus $v'' = (d\eta)(P_{p'}(p')) P_{p'} v'$
 and $w'' = (d\eta)(P_{p'}(q')) w'$.
 By Lemma~\ref{le:isometries}  below, isometries are smooth and their second derivatives are uniformly bounded. Hence the second
 derivatives of $\eta$ are uniformly bounded and we get
 \begin{equation}  \label{eq:dTM_bound_double_prime}
 |v'' - w''| \le |(d\eta)(P_{p'}(p')) P_{p'} (v' - w')| + C |p'-q'|  \, |w'| \le |v'-w'| + C |p'-q'| \, |w'|
 \end{equation}
since $d\eta(P_{p'}(p'))$ is an isometry and $P_{p'}$ has Lipschitz constant $1$.
Assertion \ref{le:distance_TM_via_imathp} implies that
$$ |p''-q''| \le d(\phi(p), \phi(q)) = d(p,q) \le C |p'-q'|.$$
Combining this with   \eqref{eq:dTM_bound_double_prime},  we get the upper bound in
\eqref{eq:quasiinvariant_dTM}, provided that $d(p,q) < \delta$.

If $d(p,q) \ge \delta$, we use the  estimate
\begin{align*}
|v''- w''| \le |v''| + |w''| = |v'| + |w'| \le |v'-w'| + 2 |w'| \le  |v'-w'| + \frac{2}{\delta} |w'|   \, d(p,q).
\end{align*}
Together with \eqref{eq:equivalence_extrinsic_distance} we get the desired conclusion.

 \medskip

\ref{le:distance_TM_via_imathw1p}:
  Set
  $$ d_0(f,g)(q) := |(\imath \circ f)(q) - (\imath \circ g)(q)|_{\R^d},
  \quad
  d_1(f,g)(q) := |d(\imath \circ f) -  d(\imath \circ g)|_{\R^d \otimes T^*_q M}.
  $$
  Let $e_1, \ldots, e_n$ be an orthonormal basis of $T_q M$. Then, by the first identity in
  \eqref{eq:MA_scalarproduct} with $B=A$
and    \eqref{eq:notation_scalar_product}
  \begin{equation*}
  d_1^2(f,g)(q) =
    \sum_\alpha |   d\imath( df e_\alpha) -
  d\imath(dg e_\alpha)|_{\R^d}^2 .
  \end{equation*}
  Thus it follows from  \eqref{eq:distance_tangent_bundle}
  that
  \begin{align}
   \label{eq:d_imath_f_underisometry}
  d_1(\psi \circ f, \psi \circ g)(q)
  \le  C(1+L) (  d_1(f,g)(q) + d_0(f,g)(q)).
  \end{align}
  Moreover, we have $d((\psi \circ f)(q), (\psi \circ g)(q)) = d(f(q), g(q))$ and thus
  \eqref{eq:equivalence_extrinsic_distance} implies that  $d_0(\psi \circ f, \psi \circ g)(q) \le
  C d_0(f,g)(q)$.
Together with  \eqref{eq:d_imath_f_underisometry} we get   \eqref{eq:dfg_quasiinvariant_isometry}.
 \end{proof}

\section{Lipschitz approximation}  \label{se:Lip_approximation}

\begin{proposition} \label{pr:lipschitz_approximation} Let $1 < p < \infty$.  There exist constants $\Lambda$, $C > 0$ with the following property.
If $f \in W^{1,p}(M; M)$ then there exists $\tilde f \in W^{1,\infty}(M;M)$ such that
\begin{equation}
|d \tilde f| \le \Lambda \quad  \text{almost everywhere}
\end{equation}
and
\begin{equation}\label{eqlipspprox2}
\int_M |\imath \circ f -\imath\circ\tilde f|^p + |d(\imath\circ f) - d(\imath\circ \tilde f)|^p \, d\vol_M \le C \int_M \dist^p(df, SO(M)) \, d\vol_M.
\end{equation}
In particular,
 \begin{equation}\int_M \dist^p(d\tilde f, SO(M)) \, d\vol_M \le C \int_M \dist^p(df, SO(M)) \, d\vol_M.
 \end{equation}
\end{proposition}

\begin{proof} This result appears as Step IV in the proof of Theorem 3  in \cite[pp. 390--392]{KupfermanMaorShachar2019}. For the convenience of the reader, we sketch the short argument.

 It suffices to show
 that there is $\eps_0>0$ such that
  the result holds when $\| \dist(df, SO(M))  \|_{L^p} < \eps_0$. Indeed, if
$\| \dist(df, SO(M))  \|_{L^p} \ge \eps_0$ we can take $\tilde f$ to be a constant map.
For  $\| \dist(df, SO(M))  \|_{L^p} < \eps_0$  the assertion follows from the corresponding result in Euclidean space (for the latter see, for example, \cite[Proposition A.1]{FrieseckeJamesMueller2002-CPAM})
by covering $M$ with finitely many charts to obtain a Lipschitz map  $\hat f$ from $M$ to $\R^d$.
Then one can use the fact that
the set where
$\imath\circ f$ and $\hat  f$ (or $d(\imath\circ f)$ and $d\hat f$) disagree is a small set when $\| \dist(df; SO(M))  \|_{L^p}$  is small.
Using that $|d\hat f| \le \Lambda$ one
concludes that $\hat f$ takes values in a small neighborhood of $\imath(M) \subset \R^d$. Thus one can define $\tilde f := \imath^{-1}\circ \pi \circ f$
where $\pi$ is the nearest point projection from a tubular neighbourhood of $\imath(M)$ to $\imath(M)$.
For the details, see, for example, Step IV in the proof of Theorem 3  in \cite[pp. 390--392]{KupfermanMaorShachar2019} or \cite [Step 1 in the proof of Theorem~4.1(i)]{KroemerMueller2021}.
\end{proof}

\section{The Piola identity and almost harmonicity}   \label{eq:almost_harmonicity}

A crucial ingredient in the proof of the Euclidean rigidity estimate is the Piola identity
$$ \Div \cof df = 0$$
which holds in the sense of distributions for maps $f: \Omega \subset \R^n \to \R^n$
which belong to the Sobolev space $W^{1,{n-1}}(\Omega; \R^n)$.
Together with the matrix estimate
$$ |F - \cof F| \le C_L  \dist(F, SO(n)) \quad \text{if $F \in \R^{n \times n}$ and $|F| \le L$}
$$
the Piola identity shows that any map with $|df| \le L$ can be written as
$u + w$  where $u$ is harmonic and $w$ is controlled by the energy in the optimal way, i.e.,
 $\|dw\|_{L^p} \le C \|\dist(df, SO(n))\|_{L^p}$.
 Being harmonic, $u$ enjoys additional regularity and compactness properties
 which allow one to reduce the rigidity estimate to an estimate for maps that are $C^1$ close to the identity.

 \medskip

 We will use a similar reasoning in the Riemannian case.
 The crucial Piola identity in this setting was shown by Kupfermann-Maor-Schachar
 \cite{KupfermanMaorShachar2019}.
 We recall the definition of the second fundamental from $\mathbb A$ in
\eqref{eq:definition_second_fundamental}  and \eqref{eq:definition_second_endos}, and that for a linear map
$F$ the cofactor map
 $\Cof F$ is defined by  \eqref{eq:define_Cof}.

 \begin{theorem}[Weak extrinsic Piola identity, \cite{KupfermanMaorShachar2019}, Theorem 2]
 \label{th:extrinsic_piola}
Assume that  $p \ge 2(n-1)$ for $n \ge 3$ and $p >2$ for $n=2$.
Let $f \in W^{1,p}(M;M)$ and set $F := \imath \circ f$. Then, for every $\xi \in W^{1,2}(M; \R^d) \cap L^\infty(M; \R^d)$,
\begin{equation}  \label{eq:extrinsic_piola}
 \int_M \langle \Cof dF, d \xi\rangle_{g, \mathfrak e} -
 \langle (\mathbb A \circ F) (\Cof dF, dF), \xi \rangle_{ \mathfrak e}
 \,  \, d\vol_M = 0.
\end{equation}
\end{theorem}
In~\eqref{eq:extrinsic_piola} $\mathfrak e$ denotes the Euclidean metric on $\R^d$ and we identify
$\Cof dF$, which is a map from $T_q M$ to $T_{F(q)} \imath(M)$,  with a map from $T_q M$ to $\R^d$.

To prove Theorem~\ref{th:extrinsic_piola} one considers the family of maps $F_t(q) := \pi (F(q) + t \xi(q))$
for $t \in (-\delta, \delta)$, where $\pi$ is the closest point projection to $\imath(M)$, which is defined
and smooth in a neighbourhood $\mathcal O$  of $\imath(M)$. To get
 \eqref{eq:extrinsic_piola}, one computes the expression
 $$  \frac{d}{dt} \int_M \Det dF_t \, d\vol_M$$
 in two different ways.

 On the one hand, the integrand is the pull-back of the volume form on $\imath(M)$. Thus the integral
 is the degree of the map $F_t$ and hence constant in $t$. On the other hand, one can commute
 differentiation and integration and use the
 identity
 \begin{align}
L(v) :=  & \,  \frac{d}{dt}|_{t=0} dF_t (v) = d \left(\frac{d}{dt}|_{t=0} F_t\right)(v) = d [(d\pi \circ F)(\xi)](v)  \nonumber \\
  = & \, (d\pi \circ F)(d\xi(v)) +  (d^2\pi \circ F)(dF(v), \xi).  \label{eq:ddt_dF_t}
 \end{align}
 With the definition of $\Cof$,  the identities $\Cof(AB) = \Cof A \, \Cof B$  as well as
 $A^T \Cof(A) = \Det A  \,  \Id$, the relation \eqref{eq:second_derivative_pi} between
 $d^2\pi$ and the second fundamental form $A$
 and the definition \eqref{eq:definition_second_endos} of $\mathbb A$ we get
 \begin{align}
 & \, \frac{d}{dt}|_{t=0} (\Det d F_t)^2 =   \frac{d}{dt}|_{t=0} \left( \Det  dF_t^T  dF_t\right)
 \nonumber  \\
 =  & \,   \langle  \Cof \left((dF)^T dF \right) ,   L^T dF + (dF)^T  L  \rangle
 \nonumber  \\
 = & \, 2 \Det dF  \langle \Cof dF, L \rangle  \nonumber \\
 = & \,  2 \Det dF  \langle \Cof dF, d\xi\rangle -2\Det dF  \langle (\mathbb A\circ F)(\Cof dF, dF),   \xi  \rangle .
 \label{eqdetdf2cof}
 \end{align}
 Thus
 $$ \frac{d}{dt}|_{t=0} \Det d F_t =  \langle \Cof dF, d\xi\rangle -  \langle (\mathbb A\circ F)(\Cof dF, dF),   \xi  \rangle.$$

This argument shows the result for smooth $f$ and $\xi$. For $f$ and $\xi$ with the   regularity stated
one can argue by approximating first $F$ and then $\xi$, see the end of Section 2.4 in
\cite{KupfermanMaorShachar2019}.
Alternatively one can  compute $\frac{d}{dt} \Det dF_t$
in local coordinates and use  dominated convergence to justify the interchange of differentiation and integration.

\medskip

In analogy with the Euclidean case, the Piola identity implies that  maps $f \in W^{1,\infty}(M;M)$
for which $\dist(df, SO(M))$ is small are almost harmonic maps.
To make this precise, we recall that a map $F: M \mapsto \imath(M)$ is a harmonic map
if it is a stationary point of $\int_M \langle dF, dF \rangle d\vol_M$ among maps with values in $\imath(M)$.
Considering variations $F_t := \pi \circ (F + t \xi)$ as above and using  \eqref{eq:ddt_dF_t},
the formula $\frac{d}{dt}|_{t=0}  \langle dF_t, dF_t\rangle = 2 \langle F, L\rangle$,  and the
relation between $d^2\pi$ and the second fundamental form,  we see that harmonic maps
satisfy
\begin{equation}  \label{eq:extrinsic_harmonic}
\int_M \langle dF, d \xi\rangle_{g, \mathfrak e} -
 \langle (\mathbb A \circ F) (dF, dF), \xi \rangle_{ \mathfrak e}
 \,  \, d\vol_M = 0
\end{equation}
for all $\xi \in (W^{1,2} \cap L^\infty)(M; \R^d)$.

\medskip

Let
\begin{eqnarray}
h' &:=& - (\mathbb A\circ F) (dF- \Cof dF, dF), \\
h &:=& dF - \Cof dF.
\end{eqnarray}
Then the Piola identity implies that  if $f \in W^{1,\infty}(M;M)$ and $F = \imath \circ f$, then
\begin{equation}  \label{eq:almost_harmonic}
\int_M \langle dF, d \xi\rangle_{g, \mathfrak e} -
 \langle (\mathbb A \circ F) (dF, dF), \xi \rangle_{g, \mathfrak e}
 \,  \, d\vol_M =  \int_M    \langle h, d\xi \rangle_{g, \mathfrak e} +  \langle h', \xi \rangle_{\mathfrak e}   \,  \, d\vol_M
\end{equation}
for all $\xi \in (W^{1,1} \cap L^\infty)(M; \R^d)$.
Moreover, it follows from \eqref{eq:CofF_equals_F} and   \eqref{eq:dist_SO} that
\begin{equation}  \label{eq:bound_error_almost_harmonic}
|h| + |h'|  \le C(\Lambda)  \,  \dist(df, SO(M)),  \quad \text{if $|df| \le \Lambda$}.
\end{equation}

To write the equation for $F$ in strong form, recall that
$\delta$   is the dual operator of the exterior differential $d: W^{1,2}(M; \R^d) \to L^2(M; \R^d \otimes T^*M)$.
In  local coordinates, $\delta$ can be expressed as follows. If  $\omega^l = \omega^l_\beta dx^\beta$ then
$$ (\delta  \omega)^l  =  - \frac{1}{\sqrt{\det g}} \partial_\alpha (  \sqrt{\det g} \, g^{\alpha \beta} \omega^l_\beta),$$ for $l=1, \ldots, d$ and $\alpha, \beta =1, \ldots, n$, with summation over repeated indices.
Set $\Delta_g := - \delta d$. Then
\begin{equation}
(\Delta_g F)^l =  \frac{1}{\sqrt{\det g}} \partial_\alpha (  \sqrt{\det g} \, g^{\alpha \beta} \partial_\beta F^l),
\end{equation}
i.e., $\Delta_g$ is the Laplace-Beltrami operator, acting componentwise.
The equation~\eqref{eq:almost_harmonic} for $F$ becomes
\begin{equation}
- \Delta_g F = \delta dF =  (\mathbb A \circ F)(dF, dF) + \delta h + h'.
\end{equation}
For future reference we recall that 
\begin{equation}   \label{eq:tension_field_tangential} 
 [-\Delta_g F - ({ \mathbb A} \circ F)(dF, dF)](p) \in T_{F(p)}\imath(M).
\end{equation}
Indeed, let $\nu$ be normal field on $\imath(M)$ and let $\eta \in C^\infty(M)$.
Let  $\xi = \eta  \, \,  (\nu \circ F)$ and let $e_1, \ldots, e_n$ be an orthonormal basis of $T_p M$. Using that $dF e_\beta \in  T_{F(p)}\imath(M)$
as well as  the first identity in \eqref{eq:MA_scalarproduct},  \eqref{eq:2nd_form_normal},  and
 \eqref{eq:definition_second_endos},
we get
\begin{align*}
& \, \langle dF, d\xi \rangle =   \langle dF,  d(\nu \circ F) \rangle \eta =
 \langle dF e_\beta, d\nu (dF e_\beta) \rangle \eta\\
= & \, 
   \langle (A\circ F)(dF e_\beta, dF e_\beta), \nu\circ F \rangle \eta
=   \langle  \mathbb A(dF, dF), \nu \circ F \rangle \eta.
\end{align*}
Since $\int_M  \langle dF, d\xi \rangle \, d\vol_M  = \int_M \langle -\Delta_g F, \nu \circ F\rangle \,  \eta \, d\vol_M$ and since
$\eta\in C^\infty(M)$ was arbitrary, the assertion follows.

\medskip

As pointed out in \cite{KupfermanMaorShachar2019},  the Piola identity implies that
 Sobolev maps with $df \in SO(M)$  are in $W^{1,\infty}$ and   harmonic maps and thus smooth.
The following lemma shows that they are actually smooth isometries, see
\cite[Theorem 1]{KupfermanMaorShachar2019}.

\begin{lemma} \label{le:isometries} Suppose that $f \in W^{1,1}(M;M)$ and $df \in SO(M)$ almost everywhere.
Then $f$ has a representative which is a smooth diffeomorphism that preserves the inner distance in $M$
and the orientation.
Moreover, the higher (covariant) derivatives of $f$ are uniformly bounded.

Conversely,  every map $\phi: M \to M$ which preserves the inner distance is smooth and $df \in SO(M)$ everywhere or $df \in O(M) \setminus SO(M)$ everywhere. Moreover, the higher (covariant) derivatives of $f$ are uniformly bounded.
\end{lemma}

\begin{proof} Let $F := \imath \circ f$. Since $df \in SO(M)$ a.e., $|df|$ is bounded and hence $F \in W^{1,\infty}(M; \R^d)$. It follows from the Piola identity and the fact that $\Cof Df = Df$ almost everywhere,  that
$F$ is a harmonic map. In particular $F$ is a weak solution of the equation
$ \Delta_g F = \mathcal (\mathbb A \circ F)(dF, dF)$.
The right hand side is in $L^\infty$. By standard elliptic estimates and induction we get $F \in W^{k,p}(M;\R^d)$
for all $k \in \N$ and $p \in (1,\infty)$. Hence $F \in C^\infty$, and therefore $f \in C^\infty$.
Since $M$ is compact, standard interior elliptic estimates
and the bound $|df|=\sqrt n$
also give uniform bounds for all derivatives, over the class of all isometries.

\medskip

Next, we show that $f: M \to M$ is bijection.  Since $f \in C^1(M;M)$ and $df \in SO(M)$ almost everywhere,
we get $df \in SO(M)$
everywhere. By the inverse function theorem $f(M)$ is open. Since $M$ is compact, the range $f(M)$ is also
compact, hence closed. Thus $f(M) = M$, since $M$ is connected.

To see that $f$ is injective, note first that $\Det df \equiv 1$ and
hence the degree agrees with the number of preimages.
Thus the area formula  gives $\vol(M) = \deg(f)  \vol (f(M)) = \deg(f)  \vol(M)$.
Hence $\deg(f) =1$
and therefore $f^{-1}(p)$ is a singleton for every $p \in M$.

\medskip

Now assume that $f$ preserves the inner distance. By the Myers-Steenrod theorem, see
\cite[Theorem 2]{Myers-Steenrod39}
$f$ is $C^1$ and  hence $df(p)$ is an isometry for every $p$.
Since $f$ is $C^1$,  we have either $df \in SO(M)$ everywhere or $df \in O(M)\setminus SO(M)$ everywhere.
In the first case, we are done. In the second case, we have $\Cof dF = - dF$ and we conclude
again that $F$ is a harmonic map and thus $F$ and $f$ are smooth with uniform bounds.
\end{proof}

\section{Harmonic map heat flow and uniform $C^{1,\alpha}$ approximation}
\label{secheatflow}

We now improve the Lipschitz approximation in Section~\ref{se:Lip_approximation} to an approximation
with uniform $C^{1,\alpha}$ bounds. This will yield compactness results in $C^1$ which in turn will allow
us to reduce the problem to bounds for the linearization of the metric deficit equation $f^*g - g = h$,
see Section~\ref{se:linearization_metric_deficit} and Section~\ref{se:proof_main} below.

To obtain approximations with  uniform $C^{1,\alpha}$ bounds, we use the extrinsic harmonic map heat flow.
We first show that for initial data with uniform Lipschitz bounds this flow exists for a fixed
time, depending only on the Lipschitz constant  of the initial datum, and satisfies uniform
$C^{1,\alpha}$ bounds for times bounded away from zero, see
  Proposition~\ref{pr:short_time_extrinsic_heat_flow}.
 If, in addition, the initial datum is an almost harmonic map in the sense of
  Section~\ref{eq:almost_harmonicity}, then
 we show that the heat flow stays $W^{1,r}$ close to the initial datum,
 see  Theorem~\ref{th:heat_flow_close_to_Phi}.

\subsection{Local existence and regularity for $W^{1,\infty}$ initial data}

\begin{proposition}  \label{pr:short_time_extrinsic_heat_flow}
Let $\Lambda > 0$, $\oldp \in (2n, \infty)$ and  $\alpha \in (0, 1- \frac{2n}\oldp)$.
 Then there exist $T_0 \in (0,1]$ and $C>0$
 such that for every
 $\Phi \in W^{1,\infty}(M;M)$ with $|d\Phi| \le \Lambda$ almost everywhere
the equation for the  extrinsic harmonic map heat flow
\begin{equation} \label{eq:extrinsic_heat_flow}
 \partial_t U^l - \Delta_g U^l  = ({\mathbb A}^l \circ U)(dU, dU) \quad \text{for $1 \le l \le d$}
\end{equation}
has a mild solution in $C^0([0, T_0]; W^{1,\oldp}(M; \R^d))$  with
$U(x,t) \in \imath(M)$ for all
$(x,t) \in M \times [0,T_0]$ and,
letting $\hat\Phi:=\imath\circ\Phi$,
\begin{eqnarray} \label{eq:extrinsic_heat_initial}
U(\cdot, 0) &=& \hat\Phi, \\
\label{eq:W1p_bound_extrinsic_heat}
\| U(\cdot, t)\|_{W^{1,\oldp}(M; \R^d)} &\le  &  C\quad \text{for all $t \in [0, T_0]$},\\
 \label{eq:C1_tau_bound}
\| U(\cdot, t)\|_{C^{1,\alpha}(M; \R^d)} &\le& C' \quad \text{for all $t \in [T_1, T_0]$}
\end{eqnarray}
where $C'$ may depend, in addition,  on $\alpha$ and $T_1 \in(0,T_0]$.
Moreover,  $U$ is a classical solution of the extrinsic heat flow for $t > 0$.
\end{proposition}

\begin{proof}  For $u  \in W^{1,\oldp}(M; \R^d)$
with $\oldp > n$ the  corresponding elliptic equation $-\Delta_g U = (\mathbb A \circ U)(dU, dU)$
is subcritical and local existence and regularity for the initial value problem
 \eqref{eq:extrinsic_heat_flow}--\eqref{eq:extrinsic_heat_initial}
 follow from standard arguments for abstract
semilinear evolution equations, see, e.g.,  \cite{pruess_simonett}.

The starting point is that the Laplace-Beltrami operator   $\Delta_g$ acting on scalar functions
  is strongly elliptic, hence sectorial on $L^\oldp(M)$ and $W^{1,\oldp}(M)$ and therefore
generates an analytic semigroup on $L^\oldp(M)$ and on $W^{1,\oldp}(M)$,  for all $\oldp \in (1,\infty)$, see for example,   \cite{pruess_simonett}, Theorems 3.3.2, 6.1.10, 6.4.3 and Remark 6.1.4.
In particular,  the semigroup $S(t) := e^{t \Delta_g}$ satisfies  for all $t \in (0,1]$
the estimates
\begin{eqnarray}      \label{eq:W2q_estimate}
\| S(t) u_0\|_{W^{2,q}} &\le& C_q  t^{-1} \|u_0\|_{L^q}  ,  \\
 \label{eq:W1p_bound}
 \| S(t) u_0\|_{W^{1,\oldp}} &\le& C_\oldp  \|u_0\|_{W^{1,\oldp}}.
\end{eqnarray}
 Moreover, by the characterization of $D((-\Delta_g)^{\oldalpha})$ or by
 using the estimates $\|S(t) u_0\|_{L^q} \le C \|u_0\|_{L^q}$,  \eqref{eq:W2q_estimate},
and the Gagliardo-Nirenberg inequality we get
\begin{align}   \label{eq:fractional_integrability}
& \| S(t) u_0\|_{W^{1,\tmpexp}} \le C_{q,\tmpexp}  t^{-\oldalpha} \| u_0\|_{L^q},
\end{align}
for
$n<q\le b\le\infty$ and
\begin{align}
  & \oldalpha := \frac12 + \frac12 \left(\frac{n}{q} - \frac{n}{\tmpexp}\right) \in (0,1).  \nonumber
\end{align}
We now fix  $\oldp \in (2n, \infty)$ and we do not indicate dependence of the constants on $\oldp$.
 We show existence of a mild  solution in $C^0([0, T_0]; W^{1,\oldp}(M; \R^d))$
 by the usual fixed point argument.
  We first  extend the action of the semigroup   $S(t) = e^{t \Delta_g}$, which  acts on scalar functions,
   to $\R^d$-valued functions
by componentwise (i.e., diagonal) action.
Then  \eqref{eq:W1p_bound} implies that for $V \in W^{1,\oldp}(M; \R^d)$ we have
\begin{equation}  \label{eq:semigroup_initial}
\| S(t) V \|_{W^{1,\oldp}} \le C \|V\|_{W^{1,\oldp}}  \quad \text{for all $t \in [0, 1]$.}
\end{equation}
Moreover, application of      \eqref{eq:fractional_integrability}
with $\tmpexp=\oldp$ and $q = \oldp/2 > n$ gives
\begin{align}  \label{eq:semigroup_rhs}
\| S(t) V\|_{W^{1,\oldp}} \le & \,  C t^{-\frac12 - \frac{n}{2\oldp}} \|V\|_{L^{\oldp/2}}  \quad  \text{for all $t \in (0, 1]$.}
\end{align}

Now let  $T_0 \in (0,1]$, set $U_s := U(\cdot, s)$ and consider the space
$$ X := \{ U \in C^0([0,T_0]; W^{1,\oldp}(M;\R^d)) : \max_{t \in [0, T_0]} \|U(\cdot, t)\|_{W^{1,\oldp}} \le  R\}.$$
We would like to reformulate \eqref{eq:extrinsic_heat_flow} as a fixed point problem in $X$. 
The second fundamental form is, however, only defined for points in   $\imath(M)$ and a priori 
the solutions of  \eqref{eq:extrinsic_heat_flow} may take values outside $\imath(M)$. 
We thus consider the closest-point projection $\pi: \mathcal O \to \imath(M)$ where $\mathcal O$ is an
open neighbourhood in $\R^d$ of the compact set $\imath(M)$. Let $\eta$ be cut-off function in $C_c^\infty(\mathcal O)$ which is $1$ on a neighbourhood of $\imath(M)$.
Then we seek a fixed point of the operator
$$ (T U)(t) := S(t) \,\hat\Phi + \int_0^t  S(t-s)  \left[ (\eta \circ U_s) (\mathbb A\circ \pi \circ U_s)(d (\pi \circ U_s), d (\pi \circ U_s)) \right]ds.$$
By definition, a fixed point of $T$ is a mild solution of the equation
$$ \partial_t U = \Delta_g U + ( \eta \circ U)  \, (\mathbb A\circ \pi \circ U)(d (\pi \circ U), d (\pi \circ U))$$
and if $U$ takes values in $\imath(M)$ then $U$ is also a mild solution of \eqref{eq:extrinsic_heat_flow} and  \eqref{eq:extrinsic_heat_initial}.

Since $M$ is compact, the quadratic forms $\mathbb A(\imath(p))$ for $p \in M$ are uniformly bounded.
Moreover, $\pi$ is smooth with uniform bounds on the support of $\eta$.
Thus
for  $U \in X$ we have
$$  \|  (\eta \circ U_s) (\mathbb A\circ \pi \circ U_s)(d (\pi \circ U_s), d (\pi \circ U_s)) \|_{L^{r/2}} \le C R^2
\quad \text{
 for $s \in (0, T_0)$.}  $$
The estimates  \eqref{eq:semigroup_initial} and  \eqref{eq:semigroup_rhs}
imply that
\begin{align*}  \| T(t) U\|_{W^{1,\oldp}}
\le & \, C \|\hat\Phi\|_{W^{1,\oldp}} + \int_0^t  C (t-s)^{-\frac12 - \frac{n}{2\oldp}}   R^2 \, ds \\
\le & \, C_1 \Lambda + C_2 T_0^{\frac12- \frac{n}{2\oldp}} R^2.
\end{align*}
Taking
$$ R := \max( 2C_1 \Lambda, 1), \quad T_0 ^{\frac12- \frac{n}{2\oldp}}  \le \frac{1}{2C_2 R}$$
we see that the operator $T$ maps $X$ to itself.

\medskip

To show that $T$ is a contraction, 
we set $N[U_s] :=  (\eta \circ  U_s) (\mathbb A\circ \pi \circ U_s)(d (\pi \circ U_s), d (\pi \circ U_s)) $
and we note that for $U, V \in X$
\begin{align*}
  \| N[U_s] - N[V_s] \|_{L^{\oldp/2}}
\le  & \,
 \,  C R^2 \|U_s - V_s\|_{L^\infty}  +  C R \|U_s-V_s\|_{W^{1,\oldp}}  \\
  \le   & \,  C R^2 \|U_s-V_s\|_{W^{1,\oldp}}
\end{align*}
since $R \ge 1$.
Set $\| U\|_{X} := \max_{t \in [0, T]} \|U_t\|_{W^{1,\oldp}}$.
It follows from \eqref{eq:semigroup_rhs} that
\begin{align*}
 \| (TU)(t) - (TV)(t)\|_{W^{1,\oldp}}
 \le  & \,  \int_0^t C R^2  (t-s)^{-\frac12 - \frac{n}{2\oldp}} \, ds \, \,  \|U-V\|_X \\
 \le & \, C R^2 T_0^{\frac12- \frac{n}{2\oldp}}  \|U-V\|_X.
\end{align*}
Thus, if in addition,
$T_0^{\frac12- \frac{n}{2\oldp}}  \le \frac1{2 C R^2}$,
then $\| TU - TV\|_X \le \frac12 \|U-V\|_X$ for all $U, V \in X$.
Hence, by the Banach fixed point theorem,  $T$ has a unique fixed point in $X$.

\medskip

To prove the $C^{1, \alpha}$ estimate    \eqref{eq:C1_tau_bound}
we use that $U(t) = (TU)(t)$ and estimate the two terms in $TU(t)$ separately.
For the first term, we use \eqref{eq:W2q_estimate}  and the Sobolev embedding to get an estimate
in $C^{1, 1-n/\oldp}$.
For the second term, we use the estimate
$$\| S( t) V\|_{C^{1,\alpha}} \le C { t}^{-\gamma} \|V\|_{L^{\frac{\oldp}2}} $$
with $\alpha = \beta  (1 - \frac{2n}{\oldp})$, $\gamma = \beta + (1-\beta) \oldalpha$,
where $\beta \in (0,1)$ and  $\oldalpha = \frac12 + \frac{n}{\oldp} \in (0,1)$.
To get this estimate we can apply \eqref{eq:W2q_estimate}
and the embedding  $W^{2,\oldp/2} \hookrightarrow C^{1, 1-2n/\oldp}$ as well as the estimate
 \eqref{eq:fractional_integrability} with $\tmpexp=\infty$ and $q= \oldp/2$.

 \medskip

By a standard bootstrap argument, using, for example, that $\Delta_g$ also generates an analytic semigroup on $C^{k,\oldalpha}$ for $\oldalpha \in (0,1)$,
we see that $U$ is smooth for $t > 0$ and hence a classical solution.

\medskip 

It only remains to show that $U$ takes values in $\imath(M)$. It suffices to show this for a time interval $[0,T_2]$
with $T_2 >0$. By a continuation argument the assertion then holds on $[0, T_0]$. 
Since $W^{1,r} \hookrightarrow C^0$ for $r > n$, we know that $U_s$ converges uniformly to $\imath \circ  \Phi$ as $s \to 0$. Thus there exists a $T_2>0$ such that $\eta \circ U_s \equiv 1$ for all $s \in [0, T_2]$.
Now set $V := \pi \circ U$ and $W := U- V$.
Then, for $t \in (0,T_2],$
\begin{equation}  \label{eq:heat_extended}  \partial_t U = \Delta_g U + (\mathbb A \circ V)(dV, dV)
= \Delta_g W + \Delta_g V + (\mathbb A \circ V)(dV, dV).
\end{equation}
Since $V$ takes values in $\imath(M)$,  we have $\partial_t V(p,t) \in T_{V(p,t)}\imath(M)$. By
 \eqref{eq:tension_field_tangential}, $(\Delta_g V + (\mathbb A \circ V)(dV, dV))(p,t) \in T_{V(p,t)}\imath(M)$ and by definition $W(p,t) \in N_{V(p,t)}\imath(M)$.
 Using  \eqref{eq:heat_extended} we get, for $t \in (0,T_2]$,
 \begin{align*}
 \langle \partial_t W, W \rangle = \langle \partial_t U, W\rangle =\langle \Delta_g W, W \rangle.
 \end{align*} 
 Hence $t  \mapsto \int_M |W_t|^2 \, d\vol_M$ is non-increasing. Moreover,  $U_t \to \imath \circ \Phi$ uniformly  as $t \to 0$ and hence
 $W_t \to 0$ uniformly. Thus $W \equiv 0$.
\end{proof}

\subsection{Refined estimates for almost harmonic initial data}

\begin{theorem}   \label{th:heat_flow_close_to_Phi}
Let $\Lambda > 0$, $ \oldr \in (1, \infty)$ and $\alpha \in (0,1)$.
Then there exist $T_1 > 0$ and $C >0$ with the following property.
If  $\Phi \in W^{1,\infty}(M;M)$ with $|d\Phi| \le \Lambda$  almost everywhere
and if there exist
 $h \in L^\oldr(M; \R^d\otimes T^*M)$ and $h' \in L^\oldr(M;\R^d)$ such that,
letting $\hat\Phi:=\imath\circ\Phi$,
$$ -\Delta_g \hat\Phi - \mathbb A\circ \hat\Phi(d\hat\Phi, d\hat\Phi)
= \delta h + h',   $$
then the extrinsic  heat flow
\begin{equation}  \label{eq:extrinsic_heat_flow2}
 \partial_t U^l - \Delta_g U^l = (\mathbb A^l \circ U)(dU, dU), \quad \text{for $1 \le l \le d$},
 \end{equation}
 has a mild solution $U \in C^0([0, 2T_1]; W^{1,\oldr}(M; \R^d))$
with $U(0,\cdot) = \hat\Phi$
such that
\begin{equation} \label{eq:W1r_estimate_heat}
\sup_{t \in [0,  2T_1]}  \| U(\cdot, t) -\hat\Phi \|_{W^{1,\oldr}} \le C (\|h\|_{L^\oldr} + \|h'\|_{L^\oldr}),
\end{equation}
\begin{equation}  \label{eq:W1r_heat_C1_onehalf}
\sup_{t \in [T_1, 2T_1]}\| U(\cdot, t)\|_{C^{1,\alpha}} \le C.
\end{equation}
\end{theorem}

\begin{proof}  Let $ \oldp \in (2n, \infty)$ so large that
\begin{equation}   \label{eq:choice_p}
  \alpha < 1 - \frac{2n}{\oldp}  \quad \text{and}  \quad  \frac1\oldp + \frac1\oldr< 1.
\end{equation}
By
Proposition~\ref{pr:short_time_extrinsic_heat_flow},  there exists
a mild solution $U \in C^0([0, T_0]; W^{1,\oldp}(M; \R^d))$ of  \eqref{eq:extrinsic_heat_flow2}
with $U(0, \cdot) = \hat\Phi$.
We set $U_s := U(s, \cdot)$  and  we will derive an integro-differential equation for $V_s := U_s - \hat\Phi$.

\medskip

By duality, the action of the semigroup $S(t)$ can be extended to $W^{-1,\oldr}$ and we get
$$ S(t) \hat\Phi - \hat\Phi = \int_0^t S'(t-s)\hat\Phi \, ds = \int_0^t  S(t-s) \Delta_g \hat\Phi\,ds.$$
Together with the definition of a mild solution we obtain
    \begin{align} V_t =  & \, \int_0^t  S(t-s)  \left[(\mathbb A \circ U_s)(dU_s, dU_s)  + \Delta_g \hat\Phi  \right] \, ds
    \nonumber  \\
= & \,  \int_0^t  S(t-s) H_s \, ds -    \int_0^t  S(t-s)  (\delta h + h')   \, ds   \label{eq:equation_hs}
\end{align}
with
\begin{align*}
H_s :=  & \,   (\mathbb A \circ U_s)(dU_s, dU_s)  - (\mathbb A\circ \hat\Phi)(d\hat\Phi, d\hat\Phi)       \\
= & \,   (\mathbb A\circ U_s)( dU_s + d\hat\Phi, dU_s - d\hat\Phi)    \\
  +    & \,        \left[(\mathbb A \circ U_s)(d\hat\Phi, d\hat\Phi) - (\mathbb A \circ \hat\Phi)(d\hat\Phi,d\hat\Phi) \right].
\end{align*}
Define $q$ by $q^{-1} = \oldr^{-1} + \oldp^{-1}$.  It follows from the choice of $\oldp$ in  \eqref{eq:choice_p}
that $q \in (1, \oldr)$.
By \eqref{eq:W1p_bound_extrinsic_heat} we have   $\sup_{s \in [0, T_0]} \|U_s\|_{W^{1,\oldp}} \le C$ and thus
$$ \|H_s  \|_{L^q}  \le   C  \|V_s\|_{W^{1,\oldr}}. $$
Thus, using    \eqref{eq:equation_hs} and  \eqref{eq:fractional_integrability}, we obtain the integro-differential inequality
\begin{equation}   \label{eq:integro_differential_ineq}
\|V_t\|_{W^{1,\oldr}} \le  C_1 \int_0^t  (t-s)^{-\oldalpha} \,   \| V_s \|_{W^{1,\oldr}} \, ds
+ C_2
\end{equation}
where
$$ \oldalpha :=  \frac12 + \frac12 \left( \frac{n}{q} - \frac{n}{\oldr}\right) = \frac12 + \frac{n}{2\oldp}  \in \left(\frac12,1\right)$$
and
\begin{equation}
C_2  :=\sup_{t \in [0, T_0]} \left\| \int_0^t  S(t-s)  (\delta h + h')   \, ds  \right\|_{W^{1,\oldr}}.
\end{equation}
If we choose  $T_1 \in (0, T_0/2]$ such that $C_1 (1-\oldalpha)^{-1} (2T_1)^{1-\oldalpha} \le \frac12$ we get
$$ \sup_{t \in [0, 2T_1]} \| V_t\|_{W^{1,\oldr}}   \le 2 C_2.$$

\medskip

Thus it suffices to show that, for all $t \in [0, 2 T_1]$,
\begin{equation}  \label{eq:bound_by_harmonic_residuum}
\left\| \int_0^t  S(t-s)  (\delta h + h')   \, ds  \right\|_{W^{1,\oldr}} \le C (\|h \|_{L^\oldr} + \|h'\|_{L^\oldr}).
\end{equation}
The estimate for the term involving $h'$ is easy. Indeed, by \eqref{eq:fractional_integrability} with $q=\tmpexp=\oldr$ we get
$\|S(t-s) h'\|_{W^{1,\oldr}} \le C(t-s)^{-1/2}  \|h'\|_{L^\oldr}$,  and we can integrate in $s$.
To estimate the term involving  $h$,
we use the fact that there
exists a unique $\Psi \in W^{1,\oldr}(M; \R^d)$ with $\int_M \Psi \, d\vol_g = 0$ such that
$\Delta_g \Psi = \delta h$ and this map $\Psi$  satisfies
\begin{equation}   \label{eq:Lp_theory_forms}  \| \Psi\|_{W^{1,\oldr}} \le C \| h\|_{L^\oldr}.
\end{equation}
This follows from the standard existence and regularity theory for the
Laplace operator acting on forms given by $\Delta = -(d\delta + \delta d)$,
see, for example \cite{morrey}, Chapters 7.3 and 7.4.
Using again the extension of the semigroup to $W^{-1,\oldr}$ we get
$$ \int_0^t S(t-s) \delta h \, ds = \int_0^t  S(t-s) \Delta_g \Psi \, ds = S(t)  \Psi - \Psi.$$
Now the estimate   \eqref{eq:bound_by_harmonic_residuum} follows from  \eqref{eq:Lp_theory_forms} and
 \eqref{eq:W1p_bound}.

 \medskip

 Finally, the estimate  \eqref{eq:W1r_heat_C1_onehalf} follows directly from   \eqref{eq:C1_tau_bound}.
 \end{proof}

\section{Linearization of the metric deficit equation and rigidity estimates close to the identity}
\label{se:linearization_metric_deficit}

For a map $f: M \to M$ we can  measure the deviation of $f$ from an isometric immersion by  the metric deficit
$$ f^*g  - g.$$
One key feature of the metric deficit is that it is invariant under the left action of isometries since
\begin{equation} \label{eq:invarianc_metric_deficit} \forall \phi \in \mathrm{Isom}(M) \quad (\phi \circ f)^*g = f^*(\phi^*g) = f^*g.
\end{equation}
If $|df|$ is bounded then it is easy to get the pointwise estimate
\begin{equation}   \label{eq:pointwise_bound_deficit}
|f^*g - g|(p) \le C \dist(df(p), SO(T_pM, T_{f(p)} M))
\end{equation}
by writing  $A = (A-Q) + Q$ in the  expression $g(Aa, Ab)$ with
$A \in \lin(T_p M, T_{f(p)} M)$ and $Q \in SO(T_pM, T_{f(p)} M)$ and optimizing over $Q$.

\medskip

Thus for the proof of our main result  is enough to show that
\begin{equation}  \label{eq:estimate_by_metric_deficit}
 \inf_{\phi\in\text{Isom}_+(M)} \| \imath \circ (\phi \circ f) - \imath \circ \id\|_{W^{1,p}} \le C   \| f^* g - g\|_{L^p}.
 \end{equation}
Note that by
\eqref{eq:dfg_quasiinvariant_isometry} the left hand side
of  \eqref{eq:estimate_by_metric_deficit}
is equivalent to    $\| \imath \circ f - \imath \circ \phi^{-1}\|_{W^{1,p}}$.

The key observation for the proof of this estimate is that the linearization (in a sense to be made precise)
of the metric deficit equation   $f^*g  - g = h$ is given by the (elliptic)
equation
\begin{equation}  \label{eq:deficit_linearized} \nabla X + (\nabla X)^T = h^\sharp
\end{equation}
where $h^\sharp$ is the $1-1$ tensor associated to the $0-2$ tensor $h$, i.e., $h(a,b) = g(h^\sharp a, b)$,
and where the vector field $X: M \to TM$ is related to $f$ by
$$ f = \exp X.$$

\medskip

To establish the desired estimate,  we  proceed as follows.
We first show that if $f$ is $C^0$ close to the identity then $f$ can be written as $f = \exp X$ and $X$ inherits
the smoothness properties of $f$, see Proposition~\ref{pr:exists_Exp_inverse} and
Lemma \ref{le:extrinsic_vs_vectorfields}.

 Next,  we make the key  observation  that  the metric deficit  $(\exp X)^* g- g$ at $p$ depends only on $X(p)$ and $\nabla X(p)$
and for $X$ small in $C^1$ is given approximately by $g(   [\nabla X + (\nabla X)^T] \cdot, \cdot)$,
see  Proposition~\ref{pr:metric_deficit}.

Thus we can hope to estimate $X$ in terms of the metric deficit $(\exp X)^* g - g$ up to solutions of
$\nabla X + (\nabla X)^T = 0$. Solutions of this equation are called Killing fields and the operator
$X \mapsto \nabla X + (\nabla X)^T$ is elliptic.
If there are no Killing fields,  then we immediately obtain an optimal  $W^{1,p}$ estimate for $X$ in terms of
$(\exp X)^* g - g$, provided that $\|X\|_{C^1}$ is sufficiently small, see  Proposition~\ref{pr:estimate_without_killing}.

If there are non-trivial Killing fields we recall the classical fact  that the space of Killing fields
is   finite dimensional and  can be viewed as the tangent space at the identity  of the Lie group
$\mathrm{Isom}(M)$ of isometries of $M$, see  Theorem~\ref{th:isom_lie}. In particular, all isometries close to the identity are generated by the flow of Killing fields at time one.

We now can use the invariance   \eqref {eq:invarianc_metric_deficit}  of the metric deficit under isometries
to obtain a new vector field $\overline X$ which generates the same metric deficit and is almost
orthogonal to all Killing fields. To do so, we minimize the $L^2$ norm  over vector fields $X_K$ defined by
$\exp X_K = \phi_K \circ \exp X$ where $\phi_K$ is the isometry generated by the flow of the Killing field
$K$, see Lemma~\ref{le:exists_minimizier_killing}.

Finally,  an easy argument by contradiction shows that the $W^{1,p}$ norm of $\overline X$ is controlled
by the $L^p$ norm of the metric deficit $(\exp X)^*g-g$, see Theorem~\ref{th:C1_rigidity_new}.
This immediately  yields the desired estimate   \eqref{eq:estimate_by_metric_deficit},
see Corollary~\ref{co:C1_rigidity_maps}.

The reasoning  ultimately rests on a) soft  arguments, based on smoothness, the proof of estimates by contradiction and compactness, and the Lie group structure of the group of isometries,  and b)  two easy calculations which exploit the commutativity of second derivatives, namely   \eqref{eq:commute_second_deficit} and
\eqref{eq:commute_derivatives_nabla_Psi_bounds}.

\subsection{From maps to vector fields}
\label{secfrommapstovecfields}
We use the fact that
a map $f:M\to M$ which is sufficiently close to the identity can be written in the form $f=\exp X$. We define the map $\Exp: TM \to M \times M$ by
\begin{equation} 
\Exp := (\pi, \exp).
\end{equation}
Thus for $v \in T_p M$ one has $\Exp \, v = (p, \exp_p v)$.

The injectivity radius $\mathrm{inj}(M)$ of the manifold $M$ is the largest value $r$
such that for each $p \in M$, the map $\exp_p$ is an embedding of the open ball or radius $r$ in $T_p M$,
see \cite{Cheeger-Ebin}, Definition 5.5.
In particular,  one has $d(p, \exp_p v) = |v|$ if $|v| < \mathrm{inj}(M)$, see also
\cite{Cheeger-Ebin}, Definition 5.5.
For  a compact  manifold $\mathrm{inj}(M) > 0$ and in fact $\mathrm{inj}(M)$ can be characterized by
two points $p$, $q \in M$ with $d(p,q) = \mathrm{inj}(M)$, see \cite{Cheeger-Ebin}, Lemma 5.6.

\begin{proposition}  \label{pr:exists_Exp_inverse}  Let $\delta \in (0, \mathrm{inj}(M)]$.
Set
$$ U_\delta := \{ v \in TM : |v| < \delta\},  \qquad D_\delta := \{(p,q) \in M \times M : d(p,q) < \delta \}.$$
Then
\begin{equation} \Exp: U_\delta \to D_\delta \quad \text{is a smooth diffeomorphism.}
\end{equation}
In particular, if $f : M \to M$ is a $C^1$ map such that $d(p,f(p)) < \mathrm{inj}(M)$  for all $p \in M$, then
$$ X := \Exp^{-1} \circ (\id, f)$$
is a $C^1$ vector field with $|X(p)| < \mathrm{inj}(M)$ and
$ f = \exp X$.
\end{proposition}

\begin{proof}    The assertion for $\Exp$ follows from the fact that $\exp_p$ is a  smooth diffeomorphism from $B_{\delta}(0) \subset T_p M$
to $B_\delta(p) \subset M$. Hence $\Exp$ is a bijective immersion from $U_\delta$ to $D_\delta$.
By the global inverse function theorem $\Exp^{-1}$ is smooth.
The assertions for $X$ then follow from the chain rule.
\end{proof}

\medskip

Recall that we fixed a smooth isometric embedding   $\imath: M \to \R^d$. For maps $f$, $g: M \to M$
we seek to estimate the distance defined in \eqref{eq:dist_W1p(M,M)},
$$  d_{1,p}(f,g) := \| \imath\circ f - \imath \circ g\|_{W^{1,p}}.$$
If $f = \exp X$ and $g = \exp Y$ then the following pointwise estimate ensures that
$d_{1,p}(f,g)$ is equivalent to $\| X- Y\|_{L^p} + \| \nabla X - \nabla Y\|_{L^p}$,
provided that $d(f(p), p) \le \frac12 \mathrm{inj}(M)$, $d(g(p),p) \le \frac12 \mathrm{inj}(M)$
and one of the functions $\imath \circ f$ or $\imath \circ g$ is in $W^{1,\infty}$.

\begin{lemma} \label{le:extrinsic_vs_vectorfields}
 Let $L>0$. There exist constants $C > 0$, $C_L > 0$ with the following property. If $\delta = \frac12 \mathrm{inj}(M)$
 and $X, Y: M \to U_\delta \subset TM$ are $C^1$ vector fields, then the maps
 $f = \exp X$ and $g = \exp Y$ satisfy
\begin{equation}  \label{le:extrinsic_pointwise}   C^{-1} |X(p) - Y(p)| \le  |(\imath \circ f)(p) - (\imath \circ g)(p)|_{\R^d} \le C |X(p) - Y(p)|.
\end{equation}
Moreover, if  $\min(|\nabla X(p)|, |\nabla Y(p)|) \le L$ then
\begin{align}
  |d (\imath \circ f)(p) - d (\imath \circ g)(p)|_{\R^d \otimes T^*_pM}
\le   & \,
  C  |\nabla X(p) - \nabla Y(p)|  \nonumber   \\
& \,   +   C_L |X(p) - Y(p)|.
   \label{le:extrinsic_pointwise_gradient}
\end{align}
If $\min(|d(\imath\circ f)(p)|, |d(\imath\circ g)(p)|) \le L$ then
\begin{align}
  |\nabla X(p) - \nabla Y(p)|   \nonumber
  \le   & \,   C \,   |d(\imath\circ f)(p) - d(\imath \circ g)(p)|_{\R^d \otimes T^*_pM}    \\
& \,   +
 C_L \,  |(\imath \circ f)(p) - (\imath \circ g)(p)|_{\R^d}.  \label{le:extrinsic_pointwise_gradient_lower}
\end{align}
\end{lemma}

For the proof of this lemma  and for later results we use the following version of the chain rule.

\begin{proposition}   \label{pr:chainrule} Let $M$ and $N$ be smooth Riemannian  manifolds.
\begin{enumerate}
\item Let $F: TM \to N$ be smooth. For $v$, $w$, $z \in T_p M$   define the horizontal derivative $F_1$ and the vertical derivative $F_2$
by
\begin{equation}
 F_1(w) v  := \frac{d}{dt}|_{t=0} F(P_t w),  \qquad
 F_2(w) z := \frac{d}{dt}|_{t=0} F(w+ t z).
  \end{equation}
Here $P_t$ denotes the parallel transport along a curve $\gamma$ with $\gamma'(0) = v$.
Then $F_1$ and $F_2$ are smooth and\begin{equation} \label{eq:chainrule_N}  d(F \circ X)(v) = F_1(X(p)) v + F_2(X(p))(\nabla_v X)
\end{equation}
for  $v \in T_p M$ and a $C^1$ vector field $X: M \to TM$.

\item Let $\Psi: TM \to TM$ be a smooth bundle map, i.e., a smooth map such that $\Psi(T_p M) \subset T_p M$.  Let
\begin{equation}
 \Psi_1(w) v  := \frac{D}{dt}|_{t=0} \Psi(P_t w),  \qquad
 \Psi_2(w) z := \frac{d}{dt}|_{t=0} \Psi(w+ t z),
  \end{equation}
  for $v$, $w$, $z \in T_p M$.
Then $\Psi_1$ and $\Psi_2$ are smooth and

\begin{equation} \label{eq:chainrule_TM}  \nabla_v (\Psi \circ X) = \Psi_1(X(p)) v + \Psi_2(X(p))(\nabla_v X)
\end{equation}
for  $v \in T_p M$ and a $C^1$ vector field $X: M \to TM$.
\end{enumerate}
\end{proposition}

\begin{proof}  This follows directly from the  fact that the Levi-Civita connection gives a splitting of the
bitangent space $T_w TM$ into a horizontal  and a vertical subspace. Alternatively, one can verify
 \eqref{eq:chainrule_N} and  \eqref{eq:chainrule_TM}  by a short calculation in local coordinates.

For the convenience of the reader, we provide some details. Let $U \subset M$ be open with $p \in U$ and let $\varphi: U \to \varphi(U) \subset \R^n$ be a chart.
Then $\phi := d\varphi : TU \to \varphi(U) \times \R^n$ is a chart for $TU$. To show  \eqref{eq:chainrule_N},
set $\tilde F := F \circ \phi^{-1}$. Denote by $\phi_1$ and $\phi_2$ the first and second component of $\phi$
and by $d_1\tilde F$ and $d_2 \tilde F$ the derivative with respect to the first and second argument.
Let $\Gamma^k_{ij}$ be the Christoffel symbols of $\phi$ and let  $\Gamma: T_p M \times T_p M \to T_pM$
be the bilinear map with
$\Gamma(\frac{\partial}{\partial x^{i}}, \frac{\partial}{\partial x^{j}}) = \Gamma^k_{ij} \frac{\partial}{\partial x^{k}}$.
Then

\begin{eqnarray}
\phi_2 (\nabla_v X) & =& d( \phi_2 \circ X)(v) + \phi_2(\Gamma(v,X(p))),   \nonumber\\
 0& =& \frac{d}{dt}|_{t=0} (\phi_2 \circ P_t X(p)) + \phi_2(\Gamma(v, X(p))),  \nonumber \\
  \label{eq:local_F1}
F_1(w)v &=& d_1 \tilde F(\phi(w)) d\phi_1(v) + d_2 \tilde F(\phi(w)) \frac{d}{dt}|_{t=0} (\phi_2 \circ P_t w),\\
\label{eq:local_F2}
F_2(w)z & = &  d_2 \tilde F(\phi(w)) \phi_2(z) = d_2 \tilde F(\phi(w)) d\phi_2(z)
\end{eqnarray}
where we used that $\phi_2$ is linear on $T_p M$. This shows that $F_1$ and $F_2$ are smooth.
Inserting $z = \nabla_v X$ and $w = X(p)$  into \eqref{eq:local_F1} and  \eqref{eq:local_F2} and adding the
resulting identities, we get  \eqref{eq:chainrule_N}, since $d(F \circ X)(v) = d(\tilde F \circ \phi(X))(v)$.

A similar calculation gives \eqref{eq:chainrule_TM}. In fact, the calculation can be simplified by
using normal coordinates at $p$. Then the Christoffel symbols vanish at $p$.
\end{proof}

\begin{proof}[Proof of Lemma~\ref{le:extrinsic_vs_vectorfields}]
Proof of   \eqref{le:extrinsic_pointwise}: Since $\imath$ is an isometric immersion and $\exp_p$ is Lipschitz on bounded sets,  we have
\begin{align*}
 |(\imath \circ f)(p) - (\imath \circ g)(p)| \le & \,  d(f(p), g(p))
 =  d(\exp_p X(p) , \exp_p Y(p)) \\
  \le & \,  C |X(p) - Y(p)|.
\end{align*}
For the lower bound we use that $\Exp\, X(p) =(p,\exp_p X(p)) = (p,f(p))$ and $\Exp\, Y(p) = (p, \exp_p Y(p)) = (p,g(p))$
and that $\Exp^{-1}$ is Lipschitz on compact subsets of $D_{\inj(M)}$.
Thus $|X(p) - Y(p)|  \le C d(f(p), g(p))$. The lower bound in  \eqref{le:extrinsic_pointwise}
now follows from  \eqref{eq:equivalence_extrinsic_distance}.

\medskip

Proof of   \eqref{le:extrinsic_pointwise_gradient}:   Since the statement is symmetric in $X$ and $Y$,
we may assume that $|\nabla Y(p)| \le L$. We apply the chain rule
\eqref{eq:chainrule_N}   to the  map $ F := \imath \circ \exp$.
This yields, for $v \in T_p M$,
\begin{equation}  \label{eq:nablaX_vs_df}
d(F \circ X)(v) = F_1(X(p)) v + F_2(X(p))(\nabla_v X)
\end{equation}
and similarly for $Y$. Thus
\begin{align}
& \,  d(\imath \circ f) v - d(\imath \circ g) v  =
F_2(X(p))(\nabla_v X - \nabla_v Y)    \nonumber  \\
 + & \,  [F_2(X(p)) - F_2(Y(p))] \nabla_v Y
+ \left(F_1(X(p)) - F_1(Y(p)) \right) v.   \label{eq:nablaXY_vs_df}
\end{align}
Now $F_1$ and $F_2$ are smooth. Thus $F_1$ is Lipschitz on compact sets and $F_2$ is bounded  and Lipschitz on
compact sets. Hence \eqref{le:extrinsic_pointwise_gradient} follows.

\medskip

Proof of \eqref{le:extrinsic_pointwise_gradient_lower}:
We first show that
\begin{equation} \label{eq:exp_estimateY}
|\nabla Y(p)| \le C_1  |d (\imath \circ g)(p) - d\imath(p)|   +  C_1 C |Y(p)|
\end{equation}
for some constants $C_1 > 0$ and $C > 0$.
The main observation is that the map $F_2(X(p)): T_p M \to T_{F(p)} \imath(M)$
in     \eqref{eq:nablaX_vs_df} is invertible if
$|X(p)| < \inj(M)$. Indeed, $F_2$ is the vertical derivative of $F$ and thus, for all $w,z \in T_pM$,
$$ F_2(w) z = d\imath(f(p)) d\exp_p(w) z$$
where $f(p) = \exp_p w$.
Now $d\imath(f(p))$ is an isometry and,  by the definition of the injectivity radius, the map
$d\exp_p(w): T_p M \to T_{f(p)} M$
is invertible for $|w| < \inj(M)$.  Since the map $TM \supset  U_{\inj(M)}
\ni w \to F_2(w)$ is smooth,  it follows by compactness that there exists $C_1 > 0$ such that
$$ |  (F_2(w))^{-1} |  \le C_1  \quad \text{if $|w| \le \delta = \frac12 \inj(M)$}.$$
Thus, exchanging the roles of $X$ and $Y$ in  \eqref{eq:nablaXY_vs_df} and setting $X=0$,
we get \eqref{eq:exp_estimateY}.

To prove \eqref{le:extrinsic_pointwise_gradient_lower}, we may assume that $|d(\imath \circ g)| \le L$
since the estimate is symmetric in $f$ and $g$.
It follows from  \eqref{eq:nablaXY_vs_df} that
\begin{align*}
& \, |\nabla X(p) - \nabla Y(p)|   \\
\le  & \,  C_1  |d (\imath \circ f)(p) - d(\imath\circ g)(p)|  + C_1 C |\nabla Y(p)| \, |X(p)-Y(p)|
+ C |X(p) - Y(p)|.
\end{align*}
Combining this estimate with  \eqref{eq:exp_estimateY} and the  estimate
$|X(p) - Y(p)| \le c |(\imath \circ f)(p) -  (\imath \circ g)(p)|$
and using that $|d\imath(p)| = \sqrt n$
and $|Y(p)| \le \delta$ we easily conclude.
\end{proof}

\subsection{The metric deficit equation and its linearization}
As mentioned above, a key observation is that
 the metric deficit  $(\exp X)^* g- g$  is local, i.e.,
 the value at $p \in M$ depends only on $X(p)$ and $\nabla X(p)$,
and for $X$ small in $C^1$ is given approximately by $g(   [\nabla X + (\nabla X)^T] \cdot, \cdot)$.

\begin{proposition}   \label{pr:metric_deficit}
\begin{enumerate}
\item There exists a smooth  map $H: TM \oplus (TM \otimes T^*M) \to T^*M \otimes T^*M$ such that
for every $C^1$  vector field  $X: M \to TM$ with $\|X\|_{C^0} \le \delta_0$,
\begin{equation}  \label{eq:metric_deficit_pointwise_X}
[(\exp X)^*g - g](p) = H(X(p), \nabla X(p)).
\end{equation}
\item
For  each  $\delta_0 > 0$ there exists a constant $C > 0$ with the following property.
For all $a,b,v \in T_p M$ and $A \in T_p M \otimes T^*_p M$ with $|v| \le \delta_0$
and $|A| \le \delta_0$,
\begin{equation}  \label{eq:metric_deficit_linearization}
|H(v, A)(a,b)  -  g( (A + A^T) a, b)|  \le C (|v| + |A|)^2 \, |a| \, |b|.
\end{equation}
\end{enumerate}
\end{proposition}
\begin{proof}
 By definition of the pullback, we have for $a, b \in T_p M$
\begin{align*} (\exp X)^*g(a,b) = g(d(\exp \circ X) a, d(\exp \circ X)b).
\end{align*}
By the chain rule (see Proposition~\ref{pr:chainrule}) there exist smooth maps $F_1$ and $F_2$ such that
$$ d(\exp \circ X) a = F_1(X(p)) a + F_2(X(p)) (\nabla_a X)$$
where $F_2$ is the vertical derivative of $\exp$ and $F_1$ is the horizontal derivative.
Therefore  the first assertion holds with
$$H(v,A)(a,b) := g( F_1(v) a + F_2(v) (A a), F_1(v) b + F_2(v) (Ab)) - g(a,b).$$
Moreover, $H(0,0) = 0$, since  $(\exp 0)^*g = \id^*g = g$.

\medskip

Since $H$ is smooth, $M$ is compact,  and $H(0,0) =0$,  the second assertion follows from the identity
\begin{equation}  \label{eq:linearization_metric_defect}
 \frac{d}{dt}|_{t=0}  H(tv,tA)(a,b) = g(A  a, b) + g(a, Ab).
 \end{equation}
To prove  \eqref{eq:linearization_metric_defect}, fix $p \in M$ and let $X: M \to TM$ be a smooth vector field such that $X(p) = v$ and $\nabla X(p) = A$.
Let $a,b \in T_p M$.
Then
\begin{align}
 & \, \frac{d}{dt}|_{t=0}  \,  H(tv,tA)(a,b) =   \frac{d}{dt}|_{t=0} \,  (\exp \circ (tX))^*g(a,b)
 \nonumber \\
 = & \,
  \frac{d}{dt}|_{t=0}  \, g(d(\exp \circ (tX))a, d(\exp \circ (tX))b)   \nonumber \\
  = & \,  \,  g\left(\frac{D}{dt}|_{t=0} d\big(\exp \circ (tX)\big) a, b\right) +  g\left(a, \frac{D}{dt}|_{t=0}d\big( (\exp \circ (tX) \big) b\right).   \label{eq:commute_second_deficit}
\end{align}
Now consider a curve $\gamma: (-\delta, \delta) \to M$ with $\gamma(0) = p$ and $\gamma'(0) = a$
and set
$$f(s,t) := (\exp \circ (tX)) \circ \gamma(s) = \exp_{\gamma(s)} \big( t X(\gamma(s))\big).$$
Computing the mixed second derivatives at $(0,0)$
and using the identity
$  \frac{D}{dt} \frac{d}{ds} =  \frac{D}{ds} \frac{d}{dt}$ (see, for example, \cite[(**), p.\ 3]{Cheeger-Ebin}
and note that $[\frac{d}{ds}, \frac{d}{dt}] = 0$) we get
\begin{align*}
  \frac{D}{dt}|_{t=0}   \,  d(\exp \circ (tX)) a =  & \,  \frac{D}{dt}|_{t=0}  \frac{d}{ds}|_{s=0} f(s,t) \\
 = & \,  \frac{D}{ds}|_{s=0}  \frac{d}{dt}|_{t=0} f(s,t)   \\
 = &  \frac{D}{ds}|_{s=0}  X(\gamma(s)) = \nabla_{a} X = A a.
\end{align*}
Repeating the calculation with $a$ replaced by $b$ we get    \eqref{eq:linearization_metric_defect}.
This concludes the proof.
\end{proof}

\subsection{Rigidity estimate if there are no nontrivial  Killing fields}

We begin with the standard linear $L^p$ estimate, which in the Euclidean setting corresponds to Korn's inequality. We only need the estimate for compact manifolds, for an estimate on domains with boundary we refer to \cite{ChenJost2002riemannian}.
A vector field $X:M\to TM$ is in the Sobolev space  $W^{1,p}$
if $X$ and the weak covariant derivative $\nabla X$ (which can be defined for example in coordinates) are in $L^p$.
\begin{lemma} \label{le:elliptic_estimate}
Let $p \in (1,\infty)$ and  define an operator $\mathcal L$ which maps  vector fields  to sections  in $TM \otimes T^*M$ by
\begin{equation}
\mathcal L X := \nabla X + (\nabla X)^T.
\end{equation}
Then there exists a constant $C>0$ which may depend on $p$ such that for every $W^{1,p}$ vector field $X$
\begin{equation}  \label{eq:W1p_estimate_Killing}
\| \nabla X\|_{L^p} \le C (\| \mathcal L X\|_{L^p} + \| X\|_{L^p}).
\end{equation}
Moreover, every vector field $X$ in $W^{1,1}$ which satisfies  $\mathcal L X = 0$  almost everywhere is smooth.
\end{lemma}

The solutions of $\mathcal L X = 0$ are called Killing fields.

\begin{proof}
The operator $\mathcal L$ is elliptic, see for example \cite[Sect.~5]{ChenJost2002riemannian}. Now both assertions follow from standard
 elliptic estimates and the compactness of $M$.
\end{proof}

\begin{proposition}   \label{pr:estimate_without_killing}
Assume that the equation  $\mathcal L X  = 0$ has
no non-zero solution. Let $p \in (1,\infty)$. Then there exists a constant $C > 0$
such that
\begin{equation}  \label{eq:W1p_estimate_Killing_improved}
 \|X\|_{W^{1,p}} \le C \|\mathcal LX\|_{L^p}.
 \end{equation}
for every $W^{1,p}$ vector field.
 Moreover,  there exist
$\delta > 0$ and $C > 0$ such that
for every vector field $X$ with $\| X\|_{C^1} \le \delta$ one has
\begin{equation}  \label{eq:rigidity_without_killing}
\| X \|_{W^{1,p}} \le C \| (\exp X)^* g - g \|_{L^p}.
\end{equation}
\end{proposition}

\begin{proof} To show  \eqref{eq:W1p_estimate_Killing_improved},
we use  \eqref{eq:W1p_estimate_Killing} and the usual argument by contradiction
based on the compact Sobolev embedding $W^{1,p} \hookrightarrow L^p$.
If  \eqref{eq:W1p_estimate_Killing_improved} does not hold, then  there exist vector fields $X_k$ such that
$\| X_k \|_{W^{1,p}} = 1$ and $\|\mathcal LX_k\|_{L^p} \to 0$.
A subsequence  of $X_k$ (not relabelled) converges weakly in $W^{1,p}$ and hence strongly in $L^p$ to a vector field
$X \in W^{1,p}$. Thus $\mathcal L X = 0$.
By assumption, this implies that $X = 0$. Hence $\| X_k\|_{L^p} \to 0$. The estimate  \eqref{eq:W1p_estimate_Killing} implies that $\| \nabla X_k \|_{L^p} \to 0$.
This contradicts the assumption $\| X_k\|_{W^{1,p}} = 1$.

\medskip

To prove  \eqref{eq:rigidity_without_killing}, we use   \eqref{eq:W1p_estimate_Killing_improved}
and the pointwise estimate
  \eqref{eq:metric_deficit_linearization}.
  This yields
  \begin{align*}
  \| X \|_{W^{1,p}} \le C \| \mathcal L  X\|_{L^p}
  \le  C_1  \| (\exp X)^* g - g \|_{L^p} +   C_2   \|X\|_{C^1} \| X\|_{W^{1,p}}.
  \end{align*}
Thus we obtain  \eqref{eq:W1p_estimate_Killing_improved} if $\delta < \frac{1}{2 C_2}$.
\end{proof}

\subsection{Isometries and Killing fields}

Let $\mathcal K$ denote the space of Killing fields, i.e., the solutions of
$\nabla X + (\nabla X)^T = 0$. It follows from  \eqref{eq:W1p_estimate_Killing}
and the compact embedding $W^{1,p} \hookrightarrow L^p$ that the space $\mathcal K$
is finite dimensional. In fact,  a geometric argument gives the optimal bound
$\dim \mathcal K \le \frac{n(n+1)}{2}$, see \cite{ONeill83}, Chapter 9, Lemma 28.

If $\mathcal K \ne \{0\}$,  then in  the estimate  \eqref{eq:W1p_estimate_Killing_improved}
the left hand side  has to be replaced by $\min_{K \in \mathcal K} \| X - K\|_{W^{1,p}}$.
With this change,  the new estimate can still  be proved by  the same argument by contradiction.
The fact that we can control $X$ only up to  a shift by elements of $\mathcal K$ is not surprising.
Indeed, it  is just the infinitesimal consequence of the invariance of the metric deficit under left composition of
$f =\exp X$ by isometries $\phi$:
$(\phi \circ f)^*g = f^*(\phi^* g) = f^* g$.

The key fact that comes to our rescue  is that  the group $\Isom(M)$ of isometries of
$M$ is a Lie group and all isometries close to the identity can be generated by the flow of a Killing field.
Recall that by $\Isom_+(M)$ we denote the subgroup of orientation preserving isometries
and note that the connected component of $\Isom(M)$ which contains the identity is contained in
$\Isom_+(M)$.

\begin{theorem} \label{th:isom_lie}
\begin{enumerate}
\item \label{th:isom_lie_finite} The group $\Isom(M)$ is a finite dimensional Lie group and the action
$\Isom(M) \times M \to M$ given by  $(\phi, p) \mapsto \phi(p)$ is smooth.
\item \label{th:isom_lie_Killing} The Lie algebra $\mathcal J(M)$ of $\Isom(M)$ can be identified with the space $\mathcal K$ of Killing fields.
\item   \label{it:diffeo_K_Isom}  Let $\mathrm{LieExp}$ be the map which maps $K \in \mathcal K$ to
$\phi_K$, the flow of $K$ at time $1$. Then $\mathrm{LieExp}$ is a diffeomorphism between an open
neighbourhood  $\mathcal U$ of $0$ in $\mathcal K$ and an open neighbourhood of the identity in
$\Isom_+(M) \subset \Isom(M)$.
\end{enumerate}
\end{theorem}

\begin{proof}
\ref{th:isom_lie_finite}: This  is the Myers-Steenrod theorem, see \cite{Myers-Steenrod39}
or \cite{ONeill83}, Chapter 9, Theorem 32.

\medskip

\ref{th:isom_lie_Killing}: By definition, $\mathcal J(M)$ consists of the left invariant vector fields of $\Isom(M)$.
Each such vector field $Y$ generates a one parameter subgroup $\psi_t$. We define
a map from $\mathcal J(M)$ to vector fields on $M$ by setting $X^+(p) := \frac{d}{dt}|_{t=0} \psi_t(p)$.
By  \cite{ONeill83}, Chapter 9, Proposition 33, the map $Y \mapsto X^+$ is a Lie anti-isomorphism
from $\mathcal J(M)$ onto $\mathcal K$.  To apply that Proposition we use    that for a compact manifold $M$ every Killing vector field is
complete, i.e., its flow is defined for all times.  Note  that  the one parameter subgroups of isometries
generated by $Y$ (as a vector field on $\Isom(M)$) and by $X^+$  (as a vector field on $M$)
are the same.

\medskip

\ref{it:diffeo_K_Isom}: This follows from \ref{th:isom_lie_Killing}. Indeed, the Lie exponential map which maps $\mathcal J(M)$ to the flow at time one
provides such a diffeomorphism, see \cite{michor08}, Theorem 4.18. Moreover, the image of the Lie exponential
map is path connected to the identity and therefore lies in $\Isom_+(M)$.
\end{proof}

\begin{remark} Note that all norms on the finite dimensional spaces  $\mathcal K$  and $\mathcal J(M)$ are equivalent.
In the following, we will use a specific norm where convenient and otherwise just write
$|K|$,  keeping  in mind that the specific constants in the subsequent estimates depend of course on the choice of norm.
\end{remark}

\medskip

We want to show  that if the metric deficit $f^*g -g$ is small, then $f$ is close to an isometry.
If $f$ is expressed as $f = \exp X$, then  Proposition~\ref{pr:metric_deficit}
suggests that $X$ is close to a Killing field. We expect that by composing $f = \exp X$ with an isometry $\phi_K$  generated
by a Killing field $K$ we can obtain a new map that is close to the identity, i.e., the corresponding vector field $Y$ is close to zero and not just close to a Killing field.

\medskip

To proceed, we first  look at the induced action of composition by  $\phi_K$ on vector fields
and then carry out the minimization.
For a Killing field $K$ we denote by $\phi_{K,t}$ its flow and we write $\phi_K := \phi_{K,1}$.

 \begin{lemma}  \label{le:composition_killing} Let $\delta: = \inj(M)$.
 Let
 $$ U_{\eta}:= \{ v \in TM : |v| < \eta\}, \quad
 B_\eta  := \{ K \in \mathcal K :  |K | < \eta\},  \quad \text{where $|K| = \|K\|_{C^0}$.} $$
 Then for $K \in B_{\delta/2}$ the map
 $$ \Psi_K : U_{\delta/2} \to TM$$
 defined by
 $$  \Psi_K := \Exp^{-1}(\pi, \phi_K \circ \exp)$$
 is a bundle map (i.e., maps $T_p M$ to $T_p M$) and the map
 $(K, v) \mapsto \Psi_K(v)$ is a  smooth map from $B_{\delta/2} \times U_{\delta/2}$ to $TM$.

 Let $X: M \to TM$ be a  $C^1$ vector field with $\|X\|_{C^0} \le \delta/4$,  let $K \in \mathcal K$ with $\|K\|_{C^0}
 \le \delta/4$  and let
 $Y = \Psi_K \circ X$. Then
 \begin{equation}  \label{eq:formula_Y_K_X} \exp Y = \phi_K \circ \exp X.
 \end{equation}
 Moreover, $Y(p)$ depends smoothly on $X(p)$ and $K$, while
   $\nabla Y(p)$ depends smoothly on $K$, $X(p)$, and $\nabla X(p)$,
 and  there exists a constant $C>0$   such that  the following estimates hold for all such vector fields $X$:
\begin{eqnarray}   \label{eq:phi_K_circ_X_zero_order} | Y - (K + X)|(p) &\le& C |K| ( |X|(p) +
 |K|), \\
    \label{eq:phi_K_circ_X} | \nabla Y - (\nabla K + \nabla X)|(p) &\le& C |K| ( |X|(p) + |\nabla X|(p) +
 |K|).
 \end{eqnarray}
  \end{lemma}

  \begin{proof}
  Let $v \in T_p M \cap U_{\delta/2}$ and $K \in B_{\delta/2}$. By the definition of $\phi_K$ we have,
  for every $q \in M$,
  $$ d(\phi_K(q), q) \le \|K\|_{C^0} < \delta/2.$$
  Moreover $d(\exp v, p) = d(\exp_p v,p) = |v| < \delta/2$ since $\delta/2 < \inj(M)$.  Thus
  $d(\phi_K (\exp v), p) < \delta$. Since $\pi(v)= p$, it follows from Proposition~\ref{pr:exists_Exp_inverse}
  that $\Exp^{-1}(\pi(v), (\phi_K \circ \exp)(v))$ is well defined. Moreover
  \begin{equation}  \label{eq:expression_PsiK}
   \Psi_K(v) = \exp_p^{-1} \phi_K(\exp v)  \in T_p M.
   \end{equation}
  Thus $\Psi_K$ is a bundle map. Smoothness of $(K,v) \mapsto \Psi_K(v)$ follows from the smoothness of $\Exp^{-1}$, smoothness of the
  map $(K,q) \mapsto \phi_K(q)$,
 and the chain rule.

 Formula  \eqref{eq:expression_PsiK} implies that
 $\exp_p \Psi_K(v) = \phi_K ( \exp_p v)$. Thus \eqref{eq:formula_Y_K_X} holds.
 Moreover $Y(p)$ depends only on $X(p)$ and $K$, since $\Psi_K$ is a bundle map.
 It follows from the chain rule  \eqref{eq:chainrule_TM} that
  $\nabla Y(p)$ depends only on $K$, $X(p)$ and $\nabla X(p)$. To see that the dependence is smooth,
  one can use the definition of the horizontal and vertical derivative $(\Psi_K)_1$ and $(\Psi_K)_2$
  and check that these functions are jointly  smooth in $K$ and its their arguments.

 \medskip

To prove  \eqref{eq:phi_K_circ_X_zero_order},  set
  $\Psi(K,v) := \Psi_K(v)$.  We denote by $D_1 \Psi$  the derivative of
  $\Psi$ with respect to the first  argument.
  We have  $\Psi(0,v) = v$ for all $v \in U_{\delta/2}$.
  Using that $d\exp_p(0) = \Id$ we also get
  $$D_1 \Psi(0,0)(K) =  \frac{d}{ds}|_{s=0} \Psi_{sK}(0) = \frac{d}{ds} |_{s=0}\exp_p^{-1} \phi_{sK}(p) = K(p).$$
  Hence
  \begin{align}
& \,  \Psi(K,v) -   (v + K(p)) = \Psi(K,v) -\Psi_0(v) -K(p)   \nonumber \\
= & \, \int_0^1  \frac{d}{ds} \Psi_{sK}(v) - D_1 \Psi(0,0)(K)   \, ds   \nonumber  \\
= & \, \int_0^1  D_1 \Psi(sK,v)(K) - D_1 \Psi(0,0)(K) \, ds  . \label{eq:expansion_Psi}
  \end{align}
  Since $\overline{U_{\delta/4}} \subset U_{\delta/2}$ and $\overline{B_{\delta/4}} \subset B_{\delta/2}$
   are
   compact   and $(K,v) \mapsto \Psi(K,v)$ is smooth it follows
    that $| D_1 \Psi(sK,v)(K) - D_1 \Psi(0,0)(K)|  \le c  |K|   (|K| + |v|)$ and this implies
  \eqref{eq:phi_K_circ_X_zero_order}.

  \medskip

  To prove   \eqref{eq:phi_K_circ_X} we argue similarly, using first the chain rule,
  Proposition~\ref{pr:chainrule}. Thus  there exist smooth maps $\Psi_1$ and $\Psi_2$
  such that, for $v \in T_p M$,
  $$ \nabla_v Y = \Psi_1(K,X(p))  v+ \Psi_2(K,X(p)) \nabla_v X.$$
  Since $\Psi(0,w) = w$ it follows that $\Psi_1(0,w) = 0$ and $\Psi_2(0,w) = \id$ for all $w \in U_{\delta/2}$.

  \medskip

  To estimate the term with $\Psi_2$ we use that  $|\Psi_2(K,w) - \Psi_2(0,w)|
  \le C |K|$ if $|w| \le \delta/4$ and $|K| \le \delta/4$ and  we get
  \begin{equation}   \label{eq:estimate_Psi2K}  |\Psi_2(K,X(p)) \nabla_v X - \nabla_v X| \le C |K|  \,  |\nabla_v X|.
  \end{equation}

  For $\Psi_1$ we can argue as in  \eqref{eq:expansion_Psi}. Thus we get  the estimate
  $ |\Psi_1(K,w) v  - \nabla_v K| \le C|K| |v|(|w| + |K|)$ if  we can show that
  \begin{equation} \label{eq:derivative_Psi1}
  \frac{d}{ds} |_{s=0}\Psi_1(sK,0) v = \nabla_v K.
  \end{equation}
  Combining the estimates for $\Psi_1$ and $\Psi_2$, we get  \eqref{eq:phi_K_circ_X}.

  \medskip

It only remains to show  \eqref{eq:derivative_Psi1}. To do so,   let $\gamma$ be a curve with $\gamma'(0) = v$ and let $P_t$ denote parallel transport along $\gamma$.
  Then
  \begin{equation*}\begin{split}
  \Psi_1(K,0) v = &\frac{D}{dt}|_{t=0} \Psi_K(P_t 0) = \frac{D}{dt}|_{t=0} \Exp^{-1}(\gamma(t),  \phi_K(\gamma(t)))\\
  = &\frac{D}{dt}|_{t=0} \exp_{\gamma(t)}^{-1} \phi_K(\gamma(t)) .
  \end{split}\end{equation*}
  Set $Y(s,t) := \exp_{\gamma(t)}^{-1} \phi_{sK}(\gamma(t))$.
  Since $d \exp_p(0)$ is the identity map,  we get
  \begin{equation} \label{eq:commute_derivatives_nabla_Psi_bounds}
   \frac{d}{ds}|_{s=0}  Y(s,t) = \frac{d}{ds}|_{s=0}  \phi_{sK}(\gamma(t)) = K(\gamma(t)).
   \end{equation}
  Since the parallel transport along $\gamma$ commutes with $\frac{d}{ds}$ and since
   $\frac{D}{dt} = P_t \circ \frac{d}{dt} \circ P_t^{-1}$, we get $\frac{d}{ds} \frac{D}{dt} Y = \frac{D}{dt} \frac{d}{ds} Y$ and thus
$$  \frac{d}{ds}|_{s=0}  \Psi_1(sK,0) v =
 \left(\frac{d}{ds} \frac{D}{dt} Y\right)(0,0) = \left(\frac{D}{dt} \frac{d}{ds} Y\right)(0,0) = \frac{D}{dt}|_{t=0}  K(\gamma(t)) = \nabla_v K.$$
  \end{proof}

\subsection{Minimizing out the action of isometries}

\begin{lemma}  \label{le:exists_minimizier_killing}   Assume that $\mathcal K \ne \{0\}$.
Let $\mathcal U$ be as in Theorem~\ref{th:isom_lie}~\ref{it:diffeo_K_Isom} and let
$$  {\widehat B}_\delta :=   \{ K \in \mathcal K : \| K \|_{L^2} < \delta \}.$$
Then
there exist $\delta_1$, $C>0$ with the following property.

If $X: M \to TM$ is a  $C^1$ vector field with $\| X\|_{C^0} \le \delta_1$
and if $\Psi_K$ is defined as in Proposition~\ref{le:composition_killing}, then the functional
$$ I_X(K) := \| \Psi_K \circ X\|_{L^2}$$
attains its minimum in the open set  $\widehat B_{\delta}$ with $\delta := 5 \|X\|_{L^2}$ and $\widehat B_\delta \subset \mathcal U$.

Moreover, if $\overline K$ is a minimizer of $I_X$ in $\widehat B_{\delta}$ with $\delta = 5 \|X\|_{L^2}$
  and $\overline X := \Psi_{\overline K} \circ X$
then, for all $K \in \mathcal K$,
\begin{equation}  \label{eq:minimizer_almost_orthogonal}
|(\overline X,  K)_{L^2} |  \le  C \| \overline X\|_{L^2}^2  \,\| K\|_{L^2}.
\end{equation}
\end{lemma}

\begin{proof} We first show existence of a minimizer in $\widehat  B_{\delta}$ for a suitable choice of $\delta > 0$.
We assume $\delta_1\le \inj(M)/4$,
pick $\delta_2>0$ such that
$\|K\|_{L^2}\le\delta_2$ implies $\|K\|_{C^0}\le\inj(M)/4$, and for $\|K\|_{L^2}\le\delta_2$
define $Y_K := \Psi_K \circ X$ as in Lemma~\ref{le:composition_killing}.
Since $Y_K(p)$
  depends smoothly  on $K$, for $\delta\le\delta_2$
the functional $I_X$ is continuous on the closed ball $\overline{B_{\delta}}:=\{ K \in \mathcal K: \|K\|_{L^2} \le\delta \}$.
Since $\mathcal K$ is finite dimensional, this ball is compact and hence $I_X$ attains its minimum
in the closed ball.  It only remains to show that, for a suitable choice of $\delta_1$ and $\delta$,
the minimizer   does not lie on the boundary of the ball.

This is an easy consequence of the pointwise estimate  \eqref{eq:phi_K_circ_X_zero_order}.
Indeed, since the norms  $\|K\|_{C^0}$ and $\|K\|_{L^2}$ are equivalent on the finite dimensional space $\mathcal K$, the estimate   \eqref{eq:phi_K_circ_X_zero_order} implies that
\begin{align*}
 \| \Psi_K \circ X\|_{L^2} \ge & \,  \| X + K \|_{L^2} - C' \|K\|_{L^2} (\|K\|_{L^2} + \|X\|_{L^2})
 \\
 \ge&\,  \|K\|_{L^2} - \|X\|_{L^2}  - C'  \|K\|_{L^2} (\|K\|_{L^2} + \|X\|_{L^2}).
 \end{align*}
Since $I_X(0) = \|X\|_{L^2}$,  a minimizer $\overline K$ of $I_X$ in the closed ball  $\overline{B_{\delta}}$
must satisfy
$$ \|   \overline K \|_{L^2} \le  2 \|X\|_{L^2}  +  C'
 \|   \overline K\|_{L^2} (\|   \overline K\|_{L^2} + \|X\|_{L^2}).$$
Now assume that $\delta$ is so small that $ 4 C' \delta  \le 1$.  Then
$$  \| \overline K\|_{L^2} \le   {\frac94} \|X\|_{L^2}  + {\frac14}  \| \overline K\|_{L^2},$$
so that $\| \overline K\|_{L^2} \le   3 \|X\|_{L^2}$.
Taking $\delta := 5\| X\|_{L^2}$, we get  $\| \overline K\|_{L^2} < \delta$.
To conclude the existence proof, we only need to show that we can choose $\delta_1$ so small
that $\delta = 5\| X\|_{L^2}$ is admissible, i.e., that we have $4 C' \delta \le 1$, $ 4 \delta_1 \le \inj(M)$ and $\delta\le\delta_2$.
By the Cauchy-Schwarz inequality we have $\|X\|_{L^2} \le (\vol(M))^{1/2} \|X \|_{C^0}$.
Thus we may take $\delta_1$ as the minimum of $\frac{1}{5} \vol(M)^{-1/2} \min(\delta_2,(4C')^{-1})$ and $\frac14 \inj(M)$.
Possibly reducing $\delta_1$ further, we can ensure
that $\widehat B_{\delta}$ lies in the open neighbourhood $\mathcal U$ in Theorem~\ref{th:isom_lie}~\ref{it:diffeo_K_Isom}.

\medskip

To prove the estimate   \eqref{eq:minimizer_almost_orthogonal}, we consider a Killing field $K$
and we first note that for sufficiently small $t$ there exists $L_t \in \widehat B_{\delta}$ such that
$\phi_{tK} \circ \phi_{\overline K} = \phi_{L_t}$. This follows from Theorem~\ref{th:isom_lie}~\ref{it:diffeo_K_Isom} and the inclusion $\widehat B_{\delta}\subset \mathcal U$.
Thus
$$ \| \overline  X\|_{L^2}^2 \le  \| \Psi_{tK} \circ \overline X\|_{L^2}^2
\quad \text{if $|t|$ is sufficiently small.} $$
Hence  $(\overline X, \,  \frac{d}{dt}|_{t=0}    \Psi_{tK} \circ \overline X)_{L^2} =0$.
It follows from
  \eqref{eq:phi_K_circ_X_zero_order}
  that we have the pointwise estimate
$$ \left| \frac{d}{dt}|_{t=0} \Psi_{tK}  \circ \overline X - K \right| \le c  |K| \, | \overline X|. $$
Using again that the $C^0$ norm and the $L^2$ norm are equivalent on $\mathcal K$, we get   \eqref{eq:minimizer_almost_orthogonal}.
\end{proof}

\subsection{Rigidity estimates modulo isometries}
We now collect the previous results to establish the desired rigidity estimates for vector fields and for maps.

\begin{theorem}  \label{th:C1_rigidity_new} Let $p\in(1,\infty)$. Let $\delta_1 > 0$ be as in
Lemma~\ref{le:exists_minimizier_killing}. Then there exist $\delta_0 \in (0, \delta_1]$  and $C>0$ with the following property. Let
$X$ be a vector field with $\|X \|_{C^1} \le \delta_0$ and let $\overline X = \Psi_{\overline K} \circ X$ be the field from
Lemma~\ref{le:exists_minimizier_killing}.
Then
\begin{equation}  \label{eq:rigidity_vectorfields}
 \| \overline X \| _{W^{1,p}} \le C \| (\exp \overline X)^*g - g\|_{L^p} = C
 \| (\exp X)^*g - g\|_{L^p} .
\end{equation}
\end{theorem}

\begin{corollary}   \label{co:C1_rigidity_maps}
Let $p \in (1,\infty)$.  Then thre exist  $\delta > 0$  and $C>0$   with the following property.
If $\| \imath \circ f - \imath\|_{C^1} \le \delta$  then there exist $\phi \in \Isom_+(M)$ such that
\begin{equation} \label{eq:C1_rigidity_maps}
\| \imath  \circ \phi \circ f - \imath\|_{W^{1,p}} \le C \|f^*g - g\|_{L^p}.
\end{equation}
If $\mathcal K = \{0\}$ then the assertion holds with $\phi = \id$. If $\mathcal K \ne \{0\}$ then
$\phi$ can be taken as the isometry generated by a Killing field $K$ with $|K| \le C \delta$.
\end{corollary}

\begin{proof} The proof consists in assembling the various results in this section.
By Proposition~\ref{pr:exists_Exp_inverse} there exists a $C^1$ vector field $X$ such that
$f = \exp X$, if $\delta \le \frac12 \inj(M)$.
Moreover, the pointwise bounds in  Lemma~\ref{le:extrinsic_vs_vectorfields} imply
that  $\|X \|_{C^1} \le C \delta$.

\medskip

Now assume first that $\mathcal K \ne \{0\}$.
After reducing $\delta$, if needed, we can apply Theorem~\ref{th:C1_rigidity_new} to obtain $\overline X = \Psi_{\overline K} \circ X$ which satisfies  \eqref{eq:rigidity_vectorfields}.
We have  $\exp \overline X = \phi_K \circ \exp X = \phi_K \circ f$.
Thus  \eqref{eq:C1_rigidity_maps} follows
from \eqref{eq:rigidity_vectorfields} and Lemma~\ref{le:extrinsic_vs_vectorfields}.

\medskip

If $\mathcal K = \{0\}$ we do not need to apply  Lemma~\ref{le:exists_minimizier_killing}.
Instead of   \eqref{eq:rigidity_vectorfields} we can directly use  \eqref{eq:rigidity_without_killing}
to estimate $\|X\|_{W^{1,p}}$ and, using
Lemma~\ref{le:extrinsic_vs_vectorfields}, we get  \eqref{eq:C1_rigidity_maps} with $\phi = \id$.
\end{proof}

\begin{proof}[Proof of Theorem~\ref{th:C1_rigidity_new}]
As in the proof of  \eqref{eq:W1p_estimate_Killing_improved} we argue by contradiction.
If the assertion is false, there exist vector fields $X_k$ such that $\|X_k\|_{C^1} \to 0$,
and the fields $\overline X_k$ constructed as in the statement from the
Killing fields  $\overline K_k$ of
Lemma~\ref{le:exists_minimizier_killing} obey
\begin{equation}  \label{eq:C1_rigidity_contradiction}
 \| (\exp \overline X_k)^*g - g\|_{L^p}   < \frac1k   \| \overline X_k \| _{W^{1,p}}.
\end{equation}
Additionally, by Lemma~\ref{le:exists_minimizier_killing} we have $\|K_k\|_{L^2}  \le 5 \|X_k\|_{L^2} \to 0$.
By Proposition   \ref{le:composition_killing} we then get $\|\overline X_k\|_{C^1} \to 0$.

\medskip

Define
$$ \eta_k := \|  \overline X_k\|_{W^{1,p}}, \quad Y_k := \frac1{\eta_k} \overline X_k.$$
Then $Y_k$ converges weakly in $W^{1,p}$ and strongly in $L^p$ to  a vector field $Y_0$.
We claim that
\begin{equation}  \label{eq:Y0_killing_perp}
 \int_M (Y_0,K)   \,d\vol_M = 0  \quad \text{for all $K \in \mathcal K$}.
 \end{equation}
 This follows directly from  \eqref{eq:minimizer_almost_orthogonal}. Indeed,
  \eqref{eq:minimizer_almost_orthogonal} implies that
 $$ |(Y_k, K)_{L^2}|  = \frac1{\eta_k}  |( \overline X_k,K)_{L^2} | \le  C  \frac1{\eta_k}  \| \overline X_k\|_{L^2}^2 \, \|K\|_{L^2}
 \le  C\|\overline X_k\|_{C^0}  \,  \|Y_k\|_{L^1}   \,  \|K\|_{L^2}.
 $$
 Since $\|  Y_k \|_{W^{1,p}} = 1$ and $\| X_k\|_{C^1} \to 0$,  we get    \eqref{eq:Y0_killing_perp}.

 \medskip

 Next, we show that
 \begin{equation}  \label{eq:Y0_killing}
  \nabla Y_0 + (\nabla Y_0)^T =0.
  \end{equation}
  By  \eqref{eq:metric_deficit_linearization}, we have the pointwise estimates
 \begin{align*}
 |\nabla \overline X_k + (\nabla \overline X_k)^T|  \le   C | (\exp \overline X_k)^* g - g|   +
C (|\overline X_k|^2  + |\nabla \overline X_k|^2).
 \end{align*}
Dividing by $\eta_k = \|\overline X_k\|_{W^{1,p}}$ and using  \eqref{eq:C1_rigidity_contradiction}
and $\|\overline X_k\|^2_{W^{1,2p}} \le \|\overline X_k\|_{C^1}  \|\overline X_k\|_{W^{1,p}}$ we see that
\begin{equation}   \label{eq:convergence_symmetric_part}
 \lim_{k \to \infty}  \|\nabla Y_k + (\nabla Y_k)^T\|_{L^p}  = 0.
\end{equation}
Since $Y_k$ converges to $Y_0$ weakly in $W^{1,p}$ and strongly in $L^p$ we get
  \eqref{eq:Y0_killing}.

  \medskip

Applying  \eqref{eq:Y0_killing_perp} with $K = Y_0$ we deduce that $Y_0 = 0$.
  Thus $Y_k  \to 0$ in $L^p$ and together with  \eqref{eq:convergence_symmetric_part}
  and the elliptic estimate  \eqref{eq:W1p_estimate_Killing} we get $\|Y_k\|_{W^{1,p}} \to 0$. This contradicts
  the assumption $\| Y_k\|_{W^{1,p}} = 1$.
  \end{proof}

\section{Proof of the main result}  \label{se:proof_main}

\begin{proof}[Proof of Theorem~\ref{th:optimal_riemannian_rigidity}]  \quad
{\it Step 1}: Set-up.\\
We recall that for $W^{1,p}$ maps $f,g: M \to M$ in~\eqref{eq:dist_W1p(M,M)} we defined the $W^{1,p}$ distance by
$$ d_{1,p}(f,g) := \| \imath \circ f - \imath \circ g\|_{W^{1,p}}$$
and we set
$$ \dist_{1,p}(f, \Isom_+(M)) := \inf_{ \phi \in \Isom_+(M)} d_{1,p}(f,\phi).$$
The infimum is actually attained since $\Isom_+(M)$ is compact, but we do not use this fact.
We also define
$$ e(f) := \| \dist(df, SO(M))\|_{L^p}
=\left( \int_M  \dist^p(df(q), SO(T_q M, T_{f(q)} M))  \, d\vol_M   \right)^{\frac1p}.$$
We need to show that there exists a constant $C_0$ such
that
\begin{equation}  \label{eq:final_estimate}
  \dist_{1,p}(f, \Isom_+(M)) \le C_0 e(f)
\end{equation}
for all $f \in W^{1,p}(M;M)$.

\medskip

{\it Step 2}: Reduction to   $\dist_{1,p}(f, \Isom_+(M)) \le 2 C_2$.\\
It is easy to see that there exist constants $C_1$ and $C_2$ such that
\begin{equation}  \label{eq:final_estimate_large_distance}
  \dist_{1,p}(f, \mathrm{Isom}_+(M)) \le C_1 e(f) + C_2.
\end{equation}
Indeed, we have the trivial estimates
$$ |\imath \circ f (q)- \imath(q)| \le d(f(q), q) \le \diam \,  M$$
and
\begin{align*}
& \, |d (\imath \circ f) (q)- d\imath(q)|  \le |d (\imath \circ f) (q)| + |d\imath(q)| = |df(q)| + \sqrt n\\
\le & \, \dist(df(q), SO(T_q M, T_{f(q)} M)) + 2 \sqrt n.
\end{align*}
Thus  \eqref{eq:final_estimate_large_distance} holds with $C_1 = 1$ and $C_2 = \diam\, M +  2 \sqrt n$.
Hence, if $  \dist_{1,p}(f, \mathrm{Isom}_+(M)) > 2 C_2$, then  \eqref{eq:final_estimate}  holds with
$C_0 = 2 C_1 = 2$.

\medskip

Therefore  it suffices to show  \eqref{eq:final_estimate} under the additional assumption that
$  \dist_{1,p}(f, \mathrm{Isom}_+(M))  \le 2 C_2$.
If the estimate does not hold under this additional assumption, then
there exist $f_k \in W^{1,p}(M;M)$ such that
\begin{equation}  \label{eq:main_theorem_contradiction}
e(f_k) \le \frac1k   \dist_{1,p}(f_k, \mathrm{Isom}_+(M))  \le \frac2k C_2.
\end{equation}

\medskip

{\it Step 3}: Approximation and compactness of low energy maps.\\
By Proposition~\ref{pr:lipschitz_approximation} there exist $\hat f_k$ such that $|d\hat f_k| \le \Lambda$
almost everywhere and
\begin{equation}  \label{eq:triangle_lipschitz}
 d_{1,p}(\hat f_k, f_k) \le C e(f_k), \qquad e(\hat f_k) \le C e(f_k).
 \end{equation}
Let $\alpha \in (0,1)$. By  Theorem~\ref{th:heat_flow_close_to_Phi},
\eqref{eq:almost_harmonic} and~\eqref{eq:bound_error_almost_harmonic}
there exist $\tilde f_k\in W^{1,p}(M;M)$ such that
\begin{equation}  \label{eq:triangle_C1delta}   d_{1,p}(\tilde f_k, \hat f_k) \le C e(f_k),  \qquad e(\tilde f_k) \le C e(f_k)
\end{equation}
and
$$  \| \imath \circ \tilde f_k\|_{C^{1,\alpha}} \le \Lambda'.$$
\medskip

We claim that the maps $\tilde f_k$ are $C^1$ close to an isometry, more precisely,
\begin{equation}  \label{eq:convergence_to_isometry_set}
\delta_k := \inf_{\phi \in \mathrm{Isom}_+(M)} \| \imath\circ  \tilde f_k- \imath \circ \phi \|_{C^1} \to 0
\quad \text{as $k \to \infty$.}
\end{equation}
Indeed, if \eqref{eq:convergence_to_isometry_set} does not hold, then
there exist a $\delta' > 0$ and a subsequence
(not relabelled) such that
\begin{equation}  \label{eq:contradiction_distance_isometries}
\inf_{\phi \in \mathrm{Isom}_+(M)} \| \imath\circ  \tilde f_k- \imath \circ \phi \|_{C^1} \ge \delta' .
\end{equation}
 By the Arzel\`a-Ascoli theorem, there exists a further subsequence (not relabelled) such that
 $F_k := \imath \circ \tilde f_k \to F$ in $C^1$. Since $\imath \circ \tilde f_k(M) \subset \imath(M)$
and $\imath(M)$ is compact,  we get $F(M) \subset \imath(M)$. Thus there exists a continuous
$f: M \to M$ such that $F = \imath \circ f$. By arguing in small charts we see that $f \in C^1$.
Since $\imath$ is an orientation preserving isometry, we have
$$  \dist(dF_k(q), SO(T_q M, d\imath(T_{\tilde f_k(q)} M))) = \dist(d\tilde f_k(q), SO(T_q M, T_{\tilde f_k(q) }M)).$$
Since $F_k \to F$ in $C^1$ and $e(\tilde f_k)  \to 0$  we conclude that
$$
dF(q) \in SO(T_q M, d \imath(  T_{f(q)} M) )\quad  \text{for all $q \in M$}
$$ and thus $df(q) \in SO(T_q M, T_{f(q)} M)$.
Thus, by Lemma~\ref{le:isometries}, we have  $f \in \mathrm{Isom}_+(M)$. Since $F_k \to \imath \circ f$ in $C^1$, this contradicts
 \eqref{eq:contradiction_distance_isometries}. This concludes the proof of   \eqref{eq:convergence_to_isometry_set}.

 \medskip

 {\it Step 4}:  Linearization of the metric deficit near the identity and conclusion.\\
 By  \eqref{eq:convergence_to_isometry_set} there exist $\phi_k \in \mathrm{Isom}_+(M)$ such that $\|  \imath \circ \tilde f_k - \imath  \circ  \phi_k\|_{C^1}
 \to 0$. Let $h_k := \phi_k^{-1} \circ \tilde  f_k$.
 Since $ |d(\imath \circ\phi_k)(p)| = \sqrt n$ for all $p \in M$, it follows
 from Lemma~\ref{le:distance_TM_via_imath}\ref{le:distance_TM_via_imathder} (applied with $\phi = \phi_k^{-1}$) that
 $\imath \circ h_k \to \imath  \quad \text{in $C^1$}$.
 Thus by Corollary~\ref{co:C1_rigidity_maps} there exists a $\tilde \phi_k \in \mathrm{Isom}_+(M)$
 such that
 $$
 d_{1,p}(\tilde\phi_k \circ h_k ,\id)
 \le C \| h_k^*g - g\|_{L^p}.$$
 Now $h_k^*g = (\phi_k^{-1} \circ \tilde f_k)^*g = (\tilde f_k)^*(\phi_k^{-1})^*g = (\tilde f_k)^*g$.
 Since  the differentials $d\tilde f_k^*$  are uniformly bounded we get from
 \eqref{eq:pointwise_bound_deficit}
 $$ \| (\tilde f_k)^*g-g\|_{L^p} \le C e(\tilde f_k).$$
 Applying Lemma~\ref{le:distance_TM_via_imath}\ref{le:distance_TM_via_imathw1p} with the isometry $(\tilde\phi_k \circ \phi_k^{-1})^{-1}$
 we finally get
 $$  d_{1,p}(\tilde f_k,  \phi_k \circ \tilde\phi_k^{-1}) \le C e(\tilde f_k) \le C e(f_k).$$
 By  \eqref{eq:triangle_lipschitz},   \eqref{eq:triangle_C1delta} and the triangle inequality we get
 $$ \dist_{1,p}(f_k, \mathrm{Isom}_+(M)) \le C e(f_k).$$
If $k$ is large  enough, this contradicts   \eqref{eq:main_theorem_contradiction},
and this contradiction concludes the proof.
\end{proof}

\addcontentsline{toc}{section}{References}
\bibliography{noneuclidean}
\bibliographystyle{acm}

\end{document}